\documentclass{article} 

\usepackage{amsthm}
\usepackage[margin=3.0cm]{geometry}
\usepackage{hyperref}       
\usepackage{url}            
\usepackage{booktabs}       
\usepackage{amsfonts}       
\usepackage{nicefrac}       
\usepackage{lipsum}
\usepackage{graphicx}
\usepackage{epstopdf}
\usepackage[utf8]{inputenc} 
\usepackage[T1]{fontenc}    
\usepackage{hyperref}       
\usepackage{url}            
\usepackage{booktabs}       
\usepackage{nicefrac}       
\usepackage[dvipsnames]{xcolor}
\usepackage{amsmath}
\usepackage{amssymb}
\usepackage{cases} 
\usepackage{bbm}
\usepackage{float}
\usepackage{todonotes}
\usepackage{graphicx}
\usepackage{subcaption}

\setuptodonotes{inline}

\usepackage{mathtools}
\usepackage{setspace}
\usepackage{newtxmath}
\usepackage{empheq}
\usepackage{caption}
\usepackage{algorithm}
\usepackage[noend]{algpseudocode}
\usepackage{subcaption}
\usepackage{amsmath, amssymb, bm} 
\usepackage{hyperref}                       
\usepackage[noabbrev,capitalize]{cleveref} 

\usepackage[shortlabels]{enumitem}
\usepackage{amsthm} 
\newtheorem{theorem}{Theorem}
\newtheorem{lemma}{Lemma}
\newtheorem{corollary}{Corollary}
\newtheorem{proposition}{Proposition}

\theoremstyle{remark}
\newtheorem{remark}{Remark}

\ifpdf
  \DeclareGraphicsExtensions{.eps,.pdf,.png,.jpg}
\else
  \DeclareGraphicsExtensions{.eps}
\fi

\DefineNamedColor{named}{CVblue}{RGB}{30, 108,  145}

\newcommand{\Our}{ADP-$\beta$-MAID}

\theoremstyle{remark}

\numberwithin{equation}{section}

\newcommand{\R}{\mathbb{R}}

\newcommand{\norm}[1]{\left\Vert #1 \right\Vert}

\renewcommand{\d}{\,\mathrm{d}}

\newcommand{\calL}{\mathcal{L}}

\newcommand{\reg}{{\mathcal J}}
\newcommand{\xdagger}{{x^\dagger}}

\newcommand{\xhat}{{\hat x}}
\newcommand{\RI}{{\R \cup \{+\infty\}}}

\newcommand{\grad}{\nabla}

\newcommand{\wsto}{\stackrel{*}{\rightharpoonup}}

\newcommand\restr[2]{{\left.\kern-\nulldelimiterspace #1 \vphantom{\big|} \right|_{#2}}}

\let\sp\relax
\newcommand{\sp}[1]{\left\langle #1 \right\rangle}
\renewcommand{\phi}{\varphi}

\newcommand{\lsc}{lower-semicontinuous}

\newcommand{\ssubset}{\subset\joinrel\subset}

\usepackage[T1]{fontenc}
\usepackage{graphicx}

\begin{document}
\title{Fast Inexact Bilevel Optimization for Analytical Deep Image Priors}

\author{
    Mohammad Sadegh Salehi\thanks{Department of Mathematical Sciences, University of Bath, Bath BA2 7AY, UK. Email: \url{{ mss226, ymk30 }@bath.ac.uk}.}
    \and  
    Tatiana A. Bubba\thanks{Department of Mathematics and Computer Science, University of Ferrara, Italy.  
    Email: \url{tatiana.bubba@unife.it}.}  
    \and  
    Yury Korolev\footnotemark[1]
}
\date{}
\maketitle       
\begin{abstract}
The analytical deep image prior (ADP) introduced by Dittmer et al.~(2020) establishes a link between deep image priors and classical regularization theory via bilevel optimization. While this is an elegant construction, it  involves expensive computations if the lower-level problem is to be solved accurately. To overcome this issue, we propose  to use adaptive inexact bilevel optimization to solve ADP problems. We discuss an extension of a recent inexact bilevel method called the method of adaptive inexact descent  of Salehi et al.~(2024) to an infinite-dimensional setting required by the ADP framework. In our numerical experiments we demonstrate that the computational speed-up achieved by adaptive inexact bilevel optimization allows one to use ADP on larger-scale problems than in the previous literature, e.g. in deblurring of 2D color images.
\vspace{0.5cm}

\end{abstract}
\section{Introduction}
The Deep Image Prior (DIP) was introduced in~\cite{ulyanov2018deep} as an unsupervized learning technique for image processing tasks such as denoising, super-resolution, and inpainting. The regularity of the reconstruction is achieved by a combination of a network architecture and early stopping during training. This method was later called \textit{regularization by architecture} in~\cite{dittmer2020regularization}. The analysis of the DIP approach has received a fair deal of interest in the literature: e.g., \cite{baguer2020computed}~interpreted it as the optimization of an analytical prior such as Total Variation; \cite{buskulic2024convergence} studied convergence of the method and derived error bounds; a Bayesian viewpoint was taken in \cite{cheng2019bayesian}; and reconstruction capabilities were analyzed in \cite{jagatap2019algorithmic}.  

Using the similarity to classical regularization methods such as Landweber iteration, \cite{dittmer2020regularization} proposed a bilevel optimization formulation of the DIP approach for an unrolled proximal gradient architecture (we postpone the precise formulation to \cref{sec:prob-stat-well-posedness}). This reformulation---albeit restricted to a particular architecture---allowed  carrying out a theoretical analysis in the spirit of regularization theory. This formulation was called the \emph{analytic deep prior} (ADP)~\cite{dittmer2020regularization}.
The ADP was further analyzed in~\cite{arndt2022regularization}, where  equivalence to classical Ivanov regularization was established. Furthermore, it was emphasized that early stopping plays a crucial role in the original DIP approach, which, if applied to the ADP, would break the equivalence with Ivanov regularization. To mimic this effect,~\cite{arndt2022regularization} proposed a Sobolev-regularized version of the ADP, cf.~\eqref{eq:ADP-Sobolev}.

Although the bilivel formulation of the ADP has theoretical advantages, it is challenging to solve numerically. Two  main approaches have been used, one based on the implicit function theorem (IFT) and another based on algorithm unrolling~\cite{adler2017solving,gregor2010learning}. While the IFT approach has convergence guarantees, in practice it requires a large number of iterations of the lower-level problem to achieve the desired accuracy, which have to be repeated at every iteration of the upper-level problem. Unrolling is faster, but lacks convergence guarantees. These alternatives were tested within the ADP framework in both~\cite{arndt2022regularization} and~\cite{dittmer2020regularization}, albeit only for simple 1D signals.

Recent advances in bilevel optimization, however, have produced methods that do not require solving the lower-level problem to high accuracy and can tolerate a controlled error in the lower-level solution while retaining convergence guarantees. These methods present a computationally efficient alternative to the IFT approach for the ADP problem. 
In this work, we propose to use the Method of Adaptive
Inexact Descent (MAID), recently introduced in~\cite{salehi2024adaptivelyinexactfirstordermethod}, for the Sobolev-regularized ADP problem~\eqref{eq:ADP-Sobolev}. On the one hand, MAID does not require solving the lower-level problem exactly and thus avoids prohibitively expensive numbers of iterations in the lower-level problem. On the other hand, it comes with convergence guarantees. The considerable computational speed-up offered by inexact bilevel optimization opens up the possibility to test ADP on more challenging problems, such as deblurring of 2D color images.

The remainder of this paper is organized as follows. In Section~\ref{sec:prob-stat-well-posedness}, we formulate the problem  and provide a summary of existing results. In \cref{sec:InexactBilevelOpt}, we recap the inexact bilevel algorithm from~\cite{salehi2024adaptivelyinexactfirstordermethod} and extend the convergence proof from the Euclidean setting to that of separable Hilbert spaces. We discuss the application of MAID to ADP and its limitations. Next, in Section~\ref{sec:experiments} we compare our results with those in~\cite{arndt2022regularization} on deblurring of 1D signals (demonstrating a significant computational speed-up) and extend our numerical experiments to various deblurring scenarios of 2D color images. 

\section{Problem statement and summary of existing results} \label{sec:prob-stat-well-posedness}
We consider linear inverse problems given by
\begin{equation}\label{eq:Ax=y}
    Ax=y,
\end{equation}
where $A \colon X \to Y$ is an injective forward operator acting between separable Hilbert spaces $X$ and $Y$. The exact right-hand side $y$ is assumed to be unavailable; as it is customary in deterministic regularization methods,  we assume that we have access to $y^\delta \in Y$ such that $\norm{y-y^\delta} \leq \delta$, where $\delta >0$ is the error magnitude. We denote by  $\xdagger$ the solution of the noise-free problem~\eqref{eq:Ax=y}.

The analytic deep prior approach (ADP)~\cite{dittmer2020regularization}  seeks to approximate the solution of~\eqref{eq:Ax=y} by a solution of the following bilevel problem
\begin{subequations}\label{eq:ADP-no-reg}
    \begin{align}
        & \min_{B \in \calL(X,Y)}  \norm{A\hat{x}(B) - y^\delta}^2 \\[-0.35em]
         & {\hat{x}(B)} \coloneqq \arg\min_{x} \norm{Bx-y^\delta}^2 + \alpha \reg(x), 
    \end{align} 
\end{subequations}
where $\reg \colon X \to \RI$ is a regularization functional chosen so that the minimizer in the lower-level problem exists and is unique, and $\alpha > 0$ a regularization parameter. The lack of coercivity of the upper-level problem makes the question of the existence of minimizers non-obvious; however, it was shown in~\cite{arndt2022regularization} that under certain conditions~\eqref{eq:ADP-no-reg} can be reformulated as a classical Ivanov regularization problem (with a modified regularization term), for which existence of solutions can be shown. Under the additional assumption that $\reg(x) = \norm{x}^2$, the authors were also able to prove stability with respect to the data $y^\delta$. The problem was also shown to be equivalent to a Tikhonov problem with a special choice of the regularization parameter.

In practice, however, solutions of~\eqref{eq:ADP-no-reg} are never computed because early stopping is used. For this reason,~\cite{arndt2022regularization} proposed the following modification of ADP 
\begin{subequations}\label{eq:ADP-op-norm}
    \begin{align}
        & \min_{B \in \calL(X,Y)}  \norm{A\hat{x}(B) - y^\delta}^2 + \beta \norm{B-A}_{\calL(X,Y)}^2 \\[-0.35em]
         & {\hat{x}(B)} \coloneqq \arg\min_{x} \norm{Bx-y^\delta}^2 + \alpha \reg(x), 
    \end{align} 
\end{subequations} 
where  $\norm{\cdot}_{\calL(Y)}$ is the operator  norm and $\beta>0$ is a regularization parameter. 
 
First, we make a note on the existence of minimizers.
The existence of a minimizer in the lower-level problem for any bounded $B$ follows from standard arguments under the usual assumptions on the regularizer $\reg$.
The situation with the upper-level problem is more difficult. 
Let $\{B_k\}$  be a minimizing sequence  of the upper-level problem. Because $\calL(X,Y)$ is the dual of a separable space (the space of nuclear operators), by the Banach-Alaoglu theorem there is a weakly-* convergent subsequence $B_k \wsto B$ (which we do not relabel). On bounded sets this convergence coincides with convergence in the weak operator topology, $\sp{B_k x,y} \to \sp{B x,y} \; \forall \, x \in X, \, y \in Y$.
To obtain existence of minimizers via the direct method of calculus of variations, we would require lower-semicontinuity of the term $\norm{A\hat{x}(B) - y^\delta}^2$ with respect to weak operator convergence, which, in turn, requires the continuity of the map $B \mapsto \hat x(B)$. This does not hold without additional assumptions, as was also noted in~\cite{arndt2022regularization}. Roughly speaking, the reason is that the term $\norm{Bx-y^\delta}^2$ is only \lsc{} and not continuous with respect to weak operator convergence.

The following modification was made in~\cite{arndt2022regularization} to ensure existence of minimizers in the upper-level problem.
Let $X=Y =L^2(\Omega)$, where $\Omega \subset \R^n$ is an open bounded Lipschitz domain, and let $A$ be an integral operator with a kernel $a \in L^2(\Omega \times \Omega)$
\begin{equation}\label{eq:B-def}
    A x (t) = \int_\Omega  a(t,\tau) x(\tau) \d \tau =: [T(a,x)](t), 
\end{equation}
where $T$ is a bilinear operator.
By restricting 
oneself to operators with Sobolev kernels $b \in H^1(\Omega \times \Omega)$, one can make use of the the compact embedding $H^1 \ssubset L^2$ (valid for any $n$). 
In~\cite{arndt2022regularization}, the author considers the following bilevel problem
\begin{subequations}\label{eq:ADP-Sobolev}
    \begin{align}
        & \min_{b \in H^1(\Omega \times \Omega)}  \norm{T(a,\hat{x}(b)) - y^\delta}_{L^2}^2 + \beta \norm{b-a}_{H^1}^2 \label{upper_Sobolev}\\[-0.35em]
         & {\hat{x}(b)} \coloneqq \arg\min_{x} \norm{T(b,x)-y^\delta}_{L^2}^2 + \alpha \reg(x),
    \end{align}
\end{subequations}
where $a \in H^1(\Omega \times \Omega)$ is the kernel of the forward operator $A$. This problem was called ADP-$\beta$ in~\cite{arndt2022regularization}. 
Existence, stability, and convergence rates for solutions of this modified problem can be shown with standard arguments, see~\cite[Sec. 3.4]{arndt2022regularization}.

\section{Inexact bilevel optimization}
\label{sec:InexactBilevelOpt}

In its standard form (i.e., not reformulated as an Ivanov or Tikhonov problem, as suggested in~\cite{arndt2022regularization}), the ADP is a bilevel optimization problem. Research has primarily focused on two main approaches: the use of the implicit function theorem (IFT) and algorithm unrolling~\cite{dittmer2020regularization}. While the IFT approach offers convergence guarantees, in practice it requires a large number of iterations of the lower-level problem to compute the exact solution $\xhat(B)$, as well as its gradient with respect to $B$. These computations must be repeated at every iteration of the upper-level problem. Achieving such high accuracy in solving the lower-level problem can be computationally prohibitive. Moreover, when approximations are used, errors in solving the lower-level problem can propagate to the gradient obtained via IFT, potentially compromising the convergence of the bilevel scheme.
On the other hand, algorithm unrolling uses a predefined number of iterations in the lower-level problem to obtain an approximation of $\xhat(B)$, with gradients computed via backpropagation. Although this approach is more practical than IFT with an exact lower-level solution, the accuracy of the approximation cannot be controlled, and convergence guarantees are lost.

Inexact bilevel optimization methods (e.g.,~\cite{pedregosa16,salehi2024adaptivelyinexactfirstordermethod}), however, can tolerate a controlled amount of error in the solution of the lower-level problem (and its gradients) while retaining convergence guarantees, hence presenting an attractive alternative to the IFT approach in ADP. In the next section we summarize the inexact bilevel method with adaptive step sizes from~\cite{salehi2024adaptivelyinexactfirstordermethod}, which we propose to use for solving the ADP-$\beta$ problem~\eqref{eq:ADP-Sobolev}. 

\subsection{Method of Adaptive Inexact Descent}
Consider the following  bilevel problem 
\begin{equation}\label{generalBilevel}
\min_{b \in H} \{g(\hat{x}(b)) +r(b)\} 
\qquad \text{s.t. }\quad 
{\hat{x}(b)} \coloneqq \arg\min_{x\in X}\ h(x,b),
\end{equation}
where $H$ and $X$ are separable Hilbert spaces and the functions $g,r$ and $h$ satisfy the following assumptions:
\begin{enumerate}
    \item[(A1)] $g$ is convex and both $g$ and $r$ are bounded below;
    \item[(A2)] $\nabla_xh$ and $\nabla_{x}^2h$ are continuous in $b$ and there exist constants $\mu(b)$ and $L(b)$, $0<\mu^*\leq\mu(b)\leq L(b)$, such that $\mu(b)I \preceq \nabla_{x}^2h(x,b) \preceq L(b) I \ \forall x$ (this implies, in particular, that the lower-level objective function is strongly convex);
    \item[(A3)]  $g$ is $L_{\nabla g}$-smooth and $r$ is $L_{\nabla r}$-smooth, which means $g$ and $r$ are continuously differentiable with $L_{\nabla g}$ and $L_{\nabla r}$ Lipschitz gradients, respectively. 
\end{enumerate}

Under (A1)-(A3),
the gradient of the upper-level objective with respect to  $b$ takes the following form, for any $x \in X$,   
\begin{equation}
    \nabla (g \circ \hat{x})(b) + \nabla r(b) = -\nabla_{xb}^2h(x,b)^T\nabla_x^2h(x,b)^{-1} \nabla g(\hat{x}(b)) + \nabla r(b),
\end{equation}
which can be obtained using the IFT.
Inexact methods avoid solving the lower-level problem exactly (which may be prohibitively expensive) and instead work with an approximation   
$\tilde{x}(b)$ such that 
$
\| {\tilde{x}(b)} - {\hat{x}(b)}\| \leq \epsilon,
$
whose error can be controlled 
by the gradient of the lower-level objective: 
\begin{equation*} 
    \|{\tilde{x}(b)} - {\hat{x}(b)}\| \leq \frac{1}{\mu(b)} \| \nabla_{x}h(\tilde{x}(b),b)\|, 
\end{equation*}
where $\mu$ is the strong convexity constant of the lower-level objective. This has been shown in~\cite{salehi2024adaptivelyinexactfirstordermethod}, but can also be seen from a simple lemma the we prove later on (\cref{lm:aposteriori_sc}).
In addition, the inversion of the Hessian can be avoided by solving  the linear system $\nabla_x^2h(\tilde{x}(b),b)q = \nabla g(\tilde{x}(b))$ up to a prescribed accuracy $\delta$, e.g., using the Conjugate Gradient (CG) method. 
Denoting the inexact hypergradient (i.e. gradient of  the upper-level problem) by \(z \coloneqq -\nabla_{x b}h(\tilde{x}(b), b)^T q + \nabla r(b)\), and the corresponding error by \(e \coloneqq z - (\nabla (g \circ \hat{x})(b) + \nabla r(b))\), we get the procedure outlined in \cref{Inexact_grad}.
\begin{algorithm}[t!]
\caption{Calculating an inexact hypergradient}\label{Inexact_grad}
\begin{algorithmic}[1]
\State{Input: $b \in H$, accuracies $\epsilon, \delta >0$.}
\Function{InexactGradient}{$b, \epsilon,\delta$}
\State{Solve lower-level problem to find $\tilde{x}(b)$ such that $\| \nabla_{x}h(\tilde{x}(b),b)\| \leq \epsilon \mu(b)$.\label{LL}}
\State{Solve $\nabla_x^2 h(\tilde{x}(b),b) q = \nabla g(\tilde{x}(b))$ with residual $\delta$.}\label{CG_step}
\State{Calculate $z = - (\nabla^2_{xb} h(\tilde{x}(b),b))^T q + \nabla r(b)$.}\label{step:Jacob}
\EndFunction
\end{algorithmic}
\end{algorithm}
\begin{algorithm*}[t!]
\caption{\cite[Algorithm 3.1]{salehi2024adaptivelyinexactfirstordermethod} Method of Adaptive Inexact Descent (MAID). Hyperparameters: {$\underline{\rho}, \underline{\nu} \in (0,1)$, $\overline{\rho}, \overline{\nu} >1$}, and  $\max_\text{BT} \in \mathbb{N}$.}\label{alg:MAID}
\begin{algorithmic}[1]
\State Input $b_0 \in H$, accuracies $\epsilon_0, \delta_0$, step size $\alpha_{0} >0$.
\For{$k = 0, 1, \dots$}
\For{$j = \max_\text{BT}, \max_\text{BT}+1, \dots$ }\label{bt_loop}
\State{$z_k, \epsilon_k, \delta_k \leftarrow$ InexactGradient($b_k, \epsilon_k, \delta_k$)} \label{updated_descend_direction}
\For{$i=0,1,\dots,j-1$}\label{inner_loop}
\If{inexact sufficient decrease $\psi(\alpha_k) \leq 0$ holds}\Comment{Cond. \eqref{psi}}
\State{Go to line \ref{gd_update_step}}\Comment{Backtracking successful}
\EndIf
\State{$\alpha_{k} \leftarrow \underline{\rho}\alpha_k$}
\Comment{Adjust the starting step size}
\EndFor
\State $\epsilon_k, \delta_k  \leftarrow \underline{\nu} \epsilon_k, \underline{\nu} \delta_k \label{BT_decrease} $  \Comment{Backtracking failed, needs higher accuracy}
\EndFor
\State{$b_{k+1} \leftarrow b_k - \alpha_{k} z_k$}\Comment{Gradient descent update}\label{gd_update_step}
\State{$\epsilon_{k+1}, \delta_{k+1}, \alpha_{k+1}  \leftarrow \overline{\nu} \epsilon_{k}, \overline{\nu} \delta_{k}, \overline{\rho} \alpha_{k}$ }\label{increase_epsilon_k}\Comment{Increasing $\epsilon_k, \delta_k, \alpha_k$}
\EndFor
\end{algorithmic}
\end{algorithm*}
Now consider the inexact descent update 
\begin{equation}\label{inexact_descent}
    b_{k+1} = b_k - \alpha_kz_k, 
\end{equation}
where $\alpha_k$ denotes the step size.
To find a suitable sequence  $\{\alpha_k\}$ and the accuracies of computing the hypergradient $\epsilon_k$ and  $\delta_k$ adaptively,~\cite{salehi2024adaptivelyinexactfirstordermethod} uses the following condition. Choose $\lambda > 0$ and, for each upper-level iteration $k = 0, 1, \dots$, set $u_k \coloneqq (\tilde{x}(b_k), b_k)$ and $\bar{\epsilon}_{k+1} {\coloneqq} \max \{\epsilon_k, \epsilon_{k+1}\}$. Let
$\ell(x(b),b) := g(x(b)) + r(b)$ and 
\begin{equation}\label{psi}
    \psi(\alpha_k) \coloneqq \ell(u_{k+1}) + \| \nabla_x \ell(u_{k+1}) \| \bar{\epsilon}_{k+1} 
    + \frac{L_{\nabla_x \ell}}{2} \  \bar{\epsilon}_{k+1}^2 - \ell(u_{k}) \\ 
    + \| \nabla_x \ell(u_{k}) \| \bar{\epsilon}_{k+1} + \lambda \alpha_k \|z_k\|^2.
\end{equation}

\begin{lemma}[Sufficient decrease condition, \cite{salehi2024adaptivelyinexactfirstordermethod}]\label{lem:suff-descent}
    Suppose that the condition $\psi(\alpha_k) \leq 0$ is satisfied at $\alpha_k$. Then the sufficient decrease in the exact upper-level function $g(\hat{x}(b_{k+1}) + r(b_{k+1}) - g(\hat{x}(b_{k}) - r(b_{k}) \leq -\lambda \alpha_k \|z_k\|^2$ holds.
\end{lemma}
In practice, since the components in \eqref{psi} depend on the accuracy, both the accuracy $\bar{\epsilon}_{k+1}$ and the step size $\alpha_k$ are decreased until the inequality holds. If the inequality is already satisfied, they are increased to reduce computational cost and take potentially larger steps.
The resulting algorithm, called the Method of Adaptive Inexact Descent (MAID), is summarized in \cref{alg:MAID}.

\begin{theorem}[convergence to a stationary point, {\cite{salehi2024adaptivelyinexactfirstordermethod}}]\label{thm:stationary-point}
    Suppose that the sub-level sets of $r(\cdot)$ from~\eqref{generalBilevel} are strongly compact in $H$. Then, under (A1)-(A3) 
    and with adaptive parameters chosen as in \cref{lem:suff-descent}, the iterates $b_k$ of \cref{alg:MAID} converge   to a stationary point $b^*$ strongly in $H$. Hence, we have 
    $$
    \lim_{k\to \infty} \|\nabla (g \circ \hat{x})(b_k) + \nabla r(b_k)\| = 0.
    $$
\end{theorem}

\begin{remark}
    The proofs of~\cite{salehi2024adaptivelyinexactfirstordermethod} are in the finite-dimensional setting where $H$ and $X$ are Euclidean spaces, but they are easily adapted to the case of separable Hilbert spaces. The only caveat is guaranteeing strong convergence of the iterates $\tilde x(b_k)$ to the solution  $\hat x(b^*)$ of the limiting lower-level problem, as the next \mbox{lemma shows.}
\end{remark}

\begin{lemma}\label{lm:aposteriori_sc}
    Let $\Phi : X \rightarrow \mathbb{R}$, where $X$ is a separable Hilbert space, be $\mu$-strongly convex and differentiable. Denote the minimizer of $\Phi$ by $x^* \in X$. Then,  we have 
    \[
    \lVert x^* - x \rVert \leq \frac{1}{\mu} \ \lVert \nabla \Phi(x) \rVert, \qquad \forall x \in X. 
    \]
\end{lemma}
\begin{proof}
Consider the Bregman distance 
$
D_\Phi(x,y) = \Phi(x) - \Phi(y) - \langle \nabla \Phi(y), x - y \rangle$, for all $x, y \in X$.
For a $\mu$-strongly convex function $\Phi$, it holds that
$\frac{\mu}{2}\lVert x - y \rVert^2 \leq D_\Phi(x,y)$.
Hence,
\begin{align*}
    \mu \lVert x - y \rVert^2 &\leq D_\Phi(x,y) + D_\Phi(y,x) = \langle \nabla \Phi(x) - \nabla \Phi(y), x - y \rangle. 
\end{align*}
Using the Cauchy--Schwarz inequality on the right-hand side, we get  
$$
{\mu}\lVert x - y \rVert^2 \leq \lVert \nabla \Phi(x) - \nabla \Phi(y) \rVert \lVert x - y \rVert,
$$  
which for $x\neq y$ yields
$  \mu \lVert x - y \rVert \leq \lVert \nabla \Phi(x) - \nabla \Phi(y) \rVert.$
Setting $y = x^*$ and noting that $\nabla \Phi(x^*) = 0$ completes the proof.
\end{proof}
\begin{remark}
Of course, \cref{lm:aposteriori_sc} remains true if $\Phi$ is only \emph{sub}-differentiable.
\end{remark}
\begin{corollary}
    Under the assumptions of \cref{thm:stationary-point}, the iterates $\tilde x(b_k)$ converge strongly to to the solution  $\hat x(b^*)$ of the limiting lower-level problem.
\end{corollary}
\begin{proof}
    It is easy to check that the minimizers of the lower-level problem converge strongly, $\hat x(b_k) \to \hat x(b^*)$. Furthermore, by \cref{lm:aposteriori_sc}, we have that 
    \begin{equation*}
        \norm{\tilde x(b_k) - \hat x(b_k)} \leq \frac1\mu \| \nabla_{x}h(\tilde{x}(b_k),b_k)\| \leq \epsilon_k \to 0. 
    \end{equation*}
    An application of the triangle inequality completes the proof.
\end{proof}

\subsection{Inexact bilevel optimization applied to ADP-$\beta$}

We now discuss briefly how the ADP-$\beta$ problem~\eqref{eq:ADP-Sobolev} fits into the general setting of~\eqref{generalBilevel}. Here we have $H=X=L^2(\Omega)$, $g(\cdot) = \norm{T(a,\cdot)-y^\delta}^2_{L^2}$, $r(\cdot) = \beta \norm{\cdot-a}_{H^1}^2$ and  $h(\cdot) = \norm{T(b,\cdot)-y^\delta}^2_{L^2} + \alpha\reg(\cdot)$. By the Sobolev embedding theorem, the sub-level sets of $r(\cdot)$ are compact in $L^2$.
The validity of (A1) is obvious. For (A2) we have the following proposition.

\begin{proposition}\label{prop:L(b)}
    Suppose that the regularizer $\reg$ is $\mu_\reg$-strongly convex, twice continuously differentiable and $\sup_x \norm{\grad_x^2 \reg(x)}_{\calL(L^2)} \leq C_\reg$ for some constants $\mu_\reg, C_\reg>0$. Then $\mu(b)$ and $L(b)$ in 
    (A2) 
    can be taken  as 
    \begin{equation*}
        \mu(b) = \alpha\mu_\reg, \quad L(b) = \norm{b}_{L^2(\Omega \times \Omega)}^2  + \alpha C_\reg.
    \end{equation*}
\end{proposition}
\begin{proof}
    The choice of $\mu(b)$ is obvious. For $L(b)$, we     
    start by computing the Hessian of the lower-level problem:
    \begin{align*}
        \grad_x \left( \norm{T(b,x)-y^\delta}_{L^2}^2 + \alpha \reg(x)\right) &= B^*Bx - B^*By^\delta + \alpha \grad_x \reg(x), \qquad \text{and}\\
        \grad_x^2 \left( \norm{T(b,x)-y^\delta}_{L^2}^2 + \alpha \reg(x)\right) &= B^*B  + \alpha \grad_x^2 \reg(x),
    \end{align*}
    where $B$ is defined analoguously to~\eqref{eq:B-def}. Taking the operator norm and a supremum in $x$ we get, with the triangle inequality, that 
    \begin{align*}
        \sup_x\norm{\grad_x^2 \left( \norm{T(b,x)-y^\delta}_{L^2}^2 + \alpha \reg(x)\right)}_{\calL(L^2)} 
        \leq  \norm{B}_{\calL(L^2)}^2 + \alpha C_\reg.
    \end{align*}
    It remains to note that $\norm{B}_{\calL(L^2)} = \norm{b}_{L^2(\Omega \times \Omega)}$.
\end{proof}
\begin{remark}
    Because of the regularization term $\norm{b-a}_{H^1}^2$ in the upper-level problem in~\eqref{eq:ADP-Sobolev}, the values $L(b_k)$  are bounded uniformly in $k$.
\end{remark}

Assumption (A3) is more problematic. The term $r(b) = \norm{b-b_0}^2_{H^1}$ is not differentiable on $L^2$ (it is only subdifferentiable), hence \cref{thm:stationary-point} does not apply. It is possible to reinterpret \cref{alg:MAID} as a subgadient method  for the upper-level objective (with a linearization of the non-convex but smooth term $\norm{T(a,\hat x(b))-y^\delta}^2_{L^2}$ at each iteration) but ensuring convergence of the subgradient method typically requires either diminishing step sizes (which is not satisfied for MAID~\cite{salehi2024adaptivelyinexactfirstordermethod}) or a uniform bound on the subgradients~\cite{shor2012minimization}, which is not satisfied in the ADP-$\beta$ problem~\eqref{eq:ADP-Sobolev}. 

Proximal methods such as PALM~\cite{bolte2014proximal} would be applicable, but, to the best of our knowledge, an adaptive inexact  version of this algorithm \`a la~\cite{salehi2024adaptivelyinexactfirstordermethod} is not available. This is subject of our future work. In our numerical experiments, which we view as a proof of concept, we stay in the finite-dimensional setting where the non-differentiability issue does not arise.

\section{Numerical Experiments}
\label{sec:experiments}
In this section, we present the numerical results of applying our method, which we call \Our{}, to the solution of the ADP-$\beta$ bilevel problem \eqref{eq:ADP-Sobolev}. First, we compare the efficiency of \Our{} 
to that of Algorithms 1 and 2 in \cite{arndt2022regularization}, on the same 1D deconvolution problem used in \cite{arndt2022regularization}, using the same settings of the original paper. We use the implementation from the supplementary material\footnote{\texttt{https://gitlab.informatik.uni-bremen.de/carndt/analytic\_deep\_prior}}
of \cite{arndt2022regularization} for Algorithms 1 and 2.

\begin{figure}[t!]
    \centering
\begin{tabular}{@{}c@{\quad}c@{\quad}c@{}}
\includegraphics[width=0.3\textwidth]{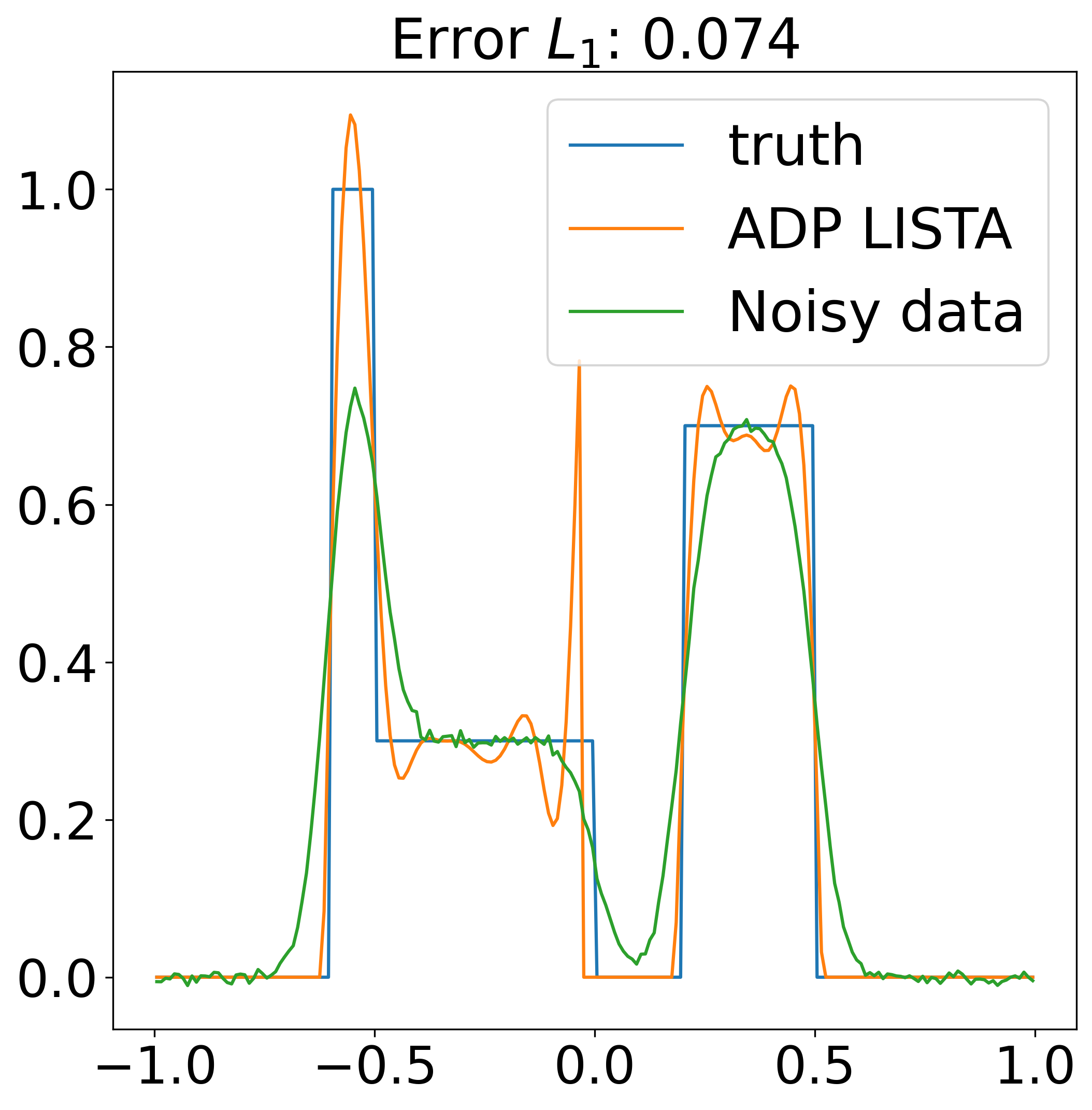}
& \includegraphics[width=0.3\textwidth]{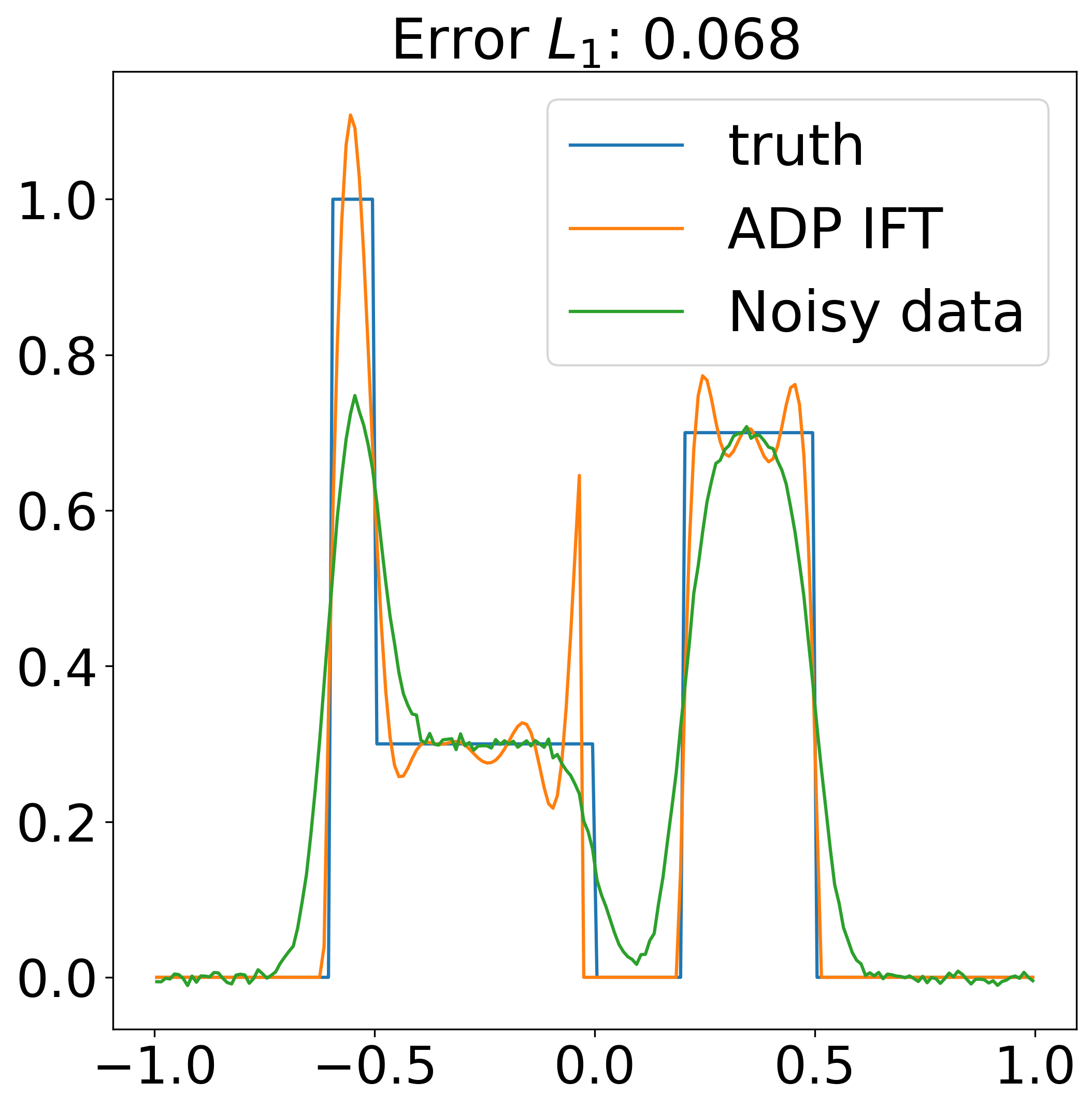}
& \includegraphics[width=0.3\textwidth]{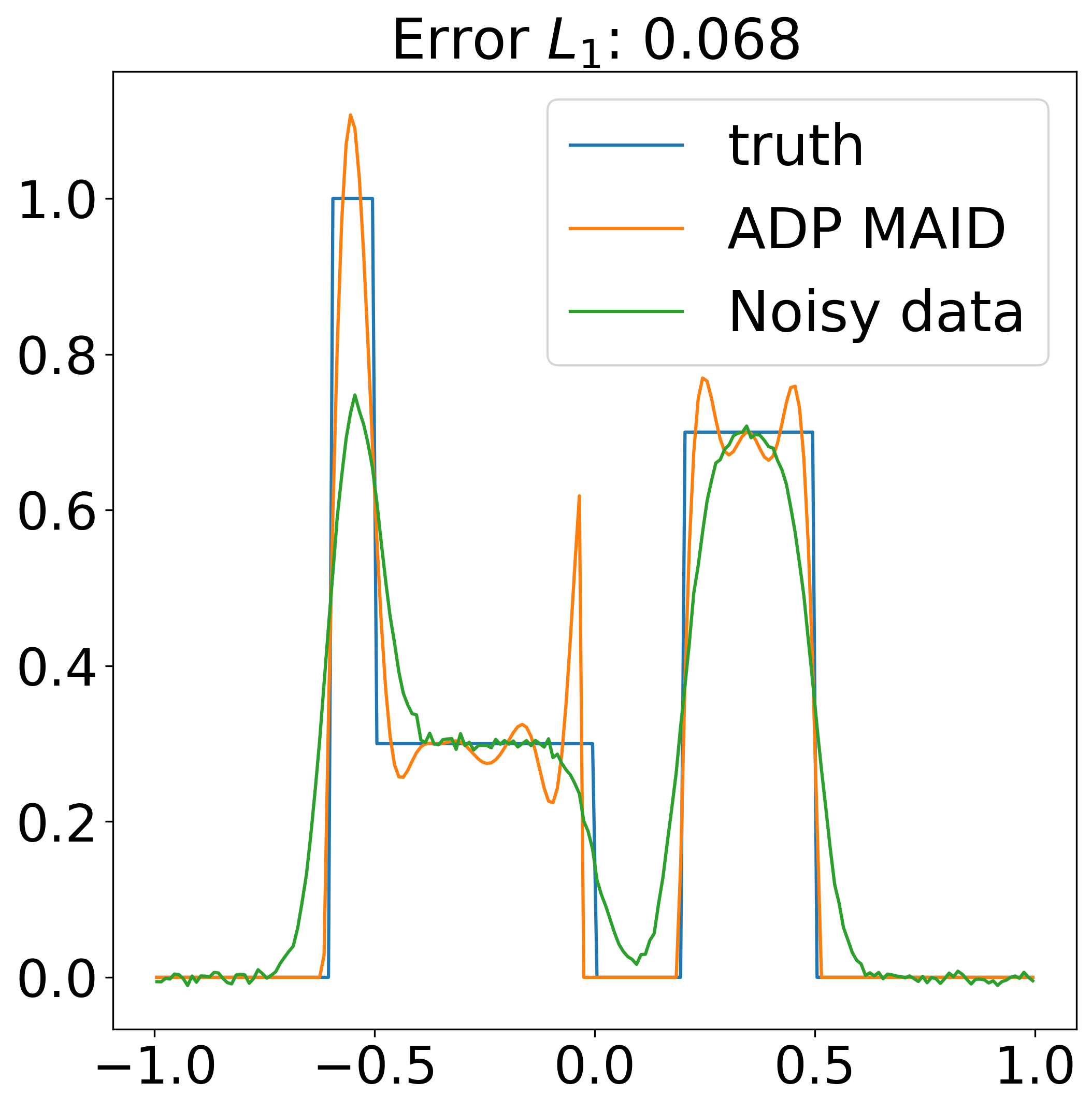}
\end{tabular}
\caption{Deblurring of a 1D signal using ADP with elastic-net regularizer, solved via ADP LISTA (left), ADP IFT (center), and \Our{} (right).} 
\label{fig:all_curve}
\end{figure}

Then, we consider deblurring and semi-blind deblurring problems for 2D color images with smoothed total variation (TV) regularization \cite{ChambolleTV}, in place of the elastic-net regularizer considered for the 1D deconvolution problem.
All experiments are conducted on an Apple Silicon M1 Pro chip using PyTorch\footnote{{Code: \url{https://github.com/MohammadSadeghSalehi/Analytical-Deep-Priors}.}}.

\subsection{1D deconvolution}\label{sec:1Ddeconvolution}
We consider a 1D signal deconvolution problem where  $A$ is the forward operator corresponding to the 1D convolution with a Gaussian kernel with standard deviation $\widetilde{\sigma} = 5$. To run the experiments, we work with the full matrix $A$ ({see \cref{fig:1D_reconst} (a)}) rather than its kernel. The corresponding blurry data are additionally corrupted with Gaussian noise with zero mean and standard deviation  $\sigma = 5 \times 10^{-3}$ (cf. green line in \cref{fig:all_curve}). 

For a fair comparison with the experiments in \cite{arndt2022regularization}, we set $\beta = 0$ in~\eqref{eq:ADP-Sobolev} and use the elastic-net regularizer $\mathcal{J}(x) = \alpha_1 \|x\|_1 + \alpha_2 \|x\|_2^2$, with $\alpha_1 = 1.2 \times 10^{-3}$ and $\alpha_2 = 4 \times 10^{-3}$ as in \cite{arndt2022regularization}, for the lower-level problem.
Algorithms 1 and 2 in \cite{arndt2022regularization}, to which we compare \Our{}, are hereafter denoted by ADP IFT and ADP LISTA\footnote{In \cite{arndt2022regularization}, the notation ``LISTA $L = \infty$'' was used to denote unrolling with a high number of iterations. We drop ``$L = \infty$'' for ease of notation.}, respectively.  

Following \cite{arndt2022regularization}, we solve the lower-level problem using the proximal-gradient descent method with step size $\lambda_x = 0.1$, and compute an approximate hypergradient. The first approach, based on IFT as in \cite{arndt2022regularization}, solves the lower-level problem with $500$ iterations to achieve high accuracy; the resulting approximate hypergradient is then used to perform inexact gradient descent. The second approach, LISTA~\cite{gregor2010learning},  solves the lower-level problem with 50 iterations, retaining the computational graph and using automatic differentiation. Additionally, {in both IFT and LISTA,} the initial solution of the lower-level problem for each successive upper-level iteration {is set to} the solution from the previous iteration to enable warm-starting.
In \Our{}, we initialize the accuracy for inexactness to $\epsilon_0 = 10^{-2}$ and set a maximum of $300$ lower-level iterations; however, this iteration limit is never reached in our experiments. The hyperparameters for \Our{} are set as $\overline{\rho} = 1.1$, $\underline{\rho} = 0.5$, $\overline{\tau} = 1.25$, and $\underline{\tau} = 0.5$. Finally, to ensure a fair comparison, the initial upper-level step size $\alpha_0 = \lambda_B = 0.1$ is the same across all three methods, and the same initialization was used for both the lower-level and upper-level problems.

As shown in \cref{fig:all_curve}, both~\Our{} and ADP IFT achieve a similar reconstruction quality with comparable learned parameters $B^*$ ({cf.~also \cref{fig:1D_reconst}}), while ADP LISTA results in a slightly higher error. However, the computational cost and CPU time, shown in \cref{fig:1D_compare_cost_time}, highlight a significant difference. In terms of CPU time, \Our{} requires approximately half the time of ADP LISTA  and only a quarter of the time required by ADP IFT with high fixed accuracy. 
Although ADP LISTA converged to a suboptimal stationary point, possibly due to insufficient lower-level accuracy in early upper-level iterations, \Our{} {reached a better stationary point} with significantly less computational cost than ADP IFT. 
{\cref{fig:1D_reconst} shows the initial convolutional operator $A=B_0$, along with the convolutional operator $B^*$ obtained by solving the ADP-$\beta$ problem applied to 1D deconvolution, for all three methods, corresponding to the reconstructions in \cref{fig:all_curve}.}

\begin{figure}[h!]
 \centering
    \begin{tabular}{@{}c@{\quad}c@{\quad}c@{\quad}c@{}}
\includegraphics[width=0.23\textwidth]{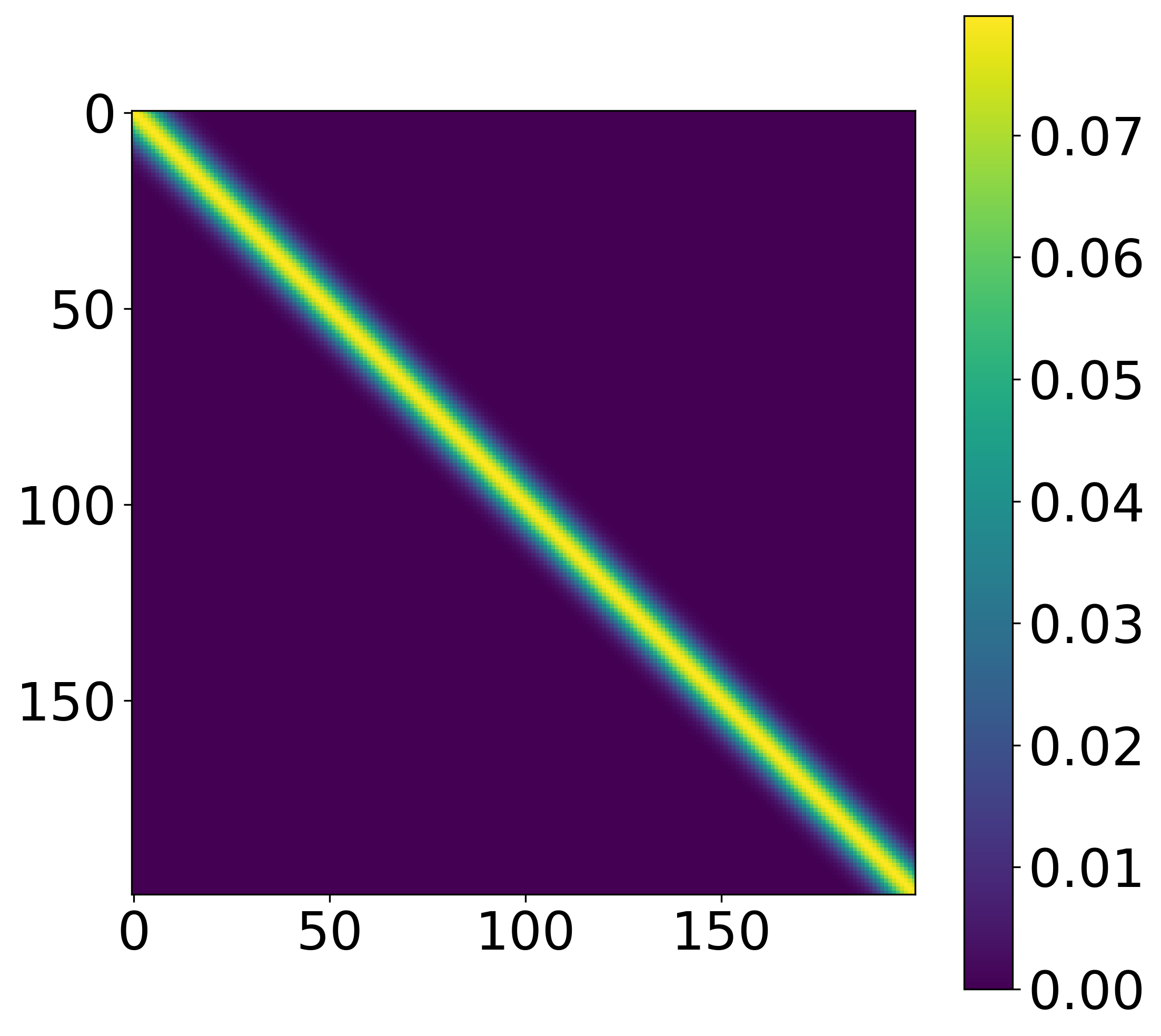}\label{full_matrix}
& \includegraphics[width=0.23\textwidth]{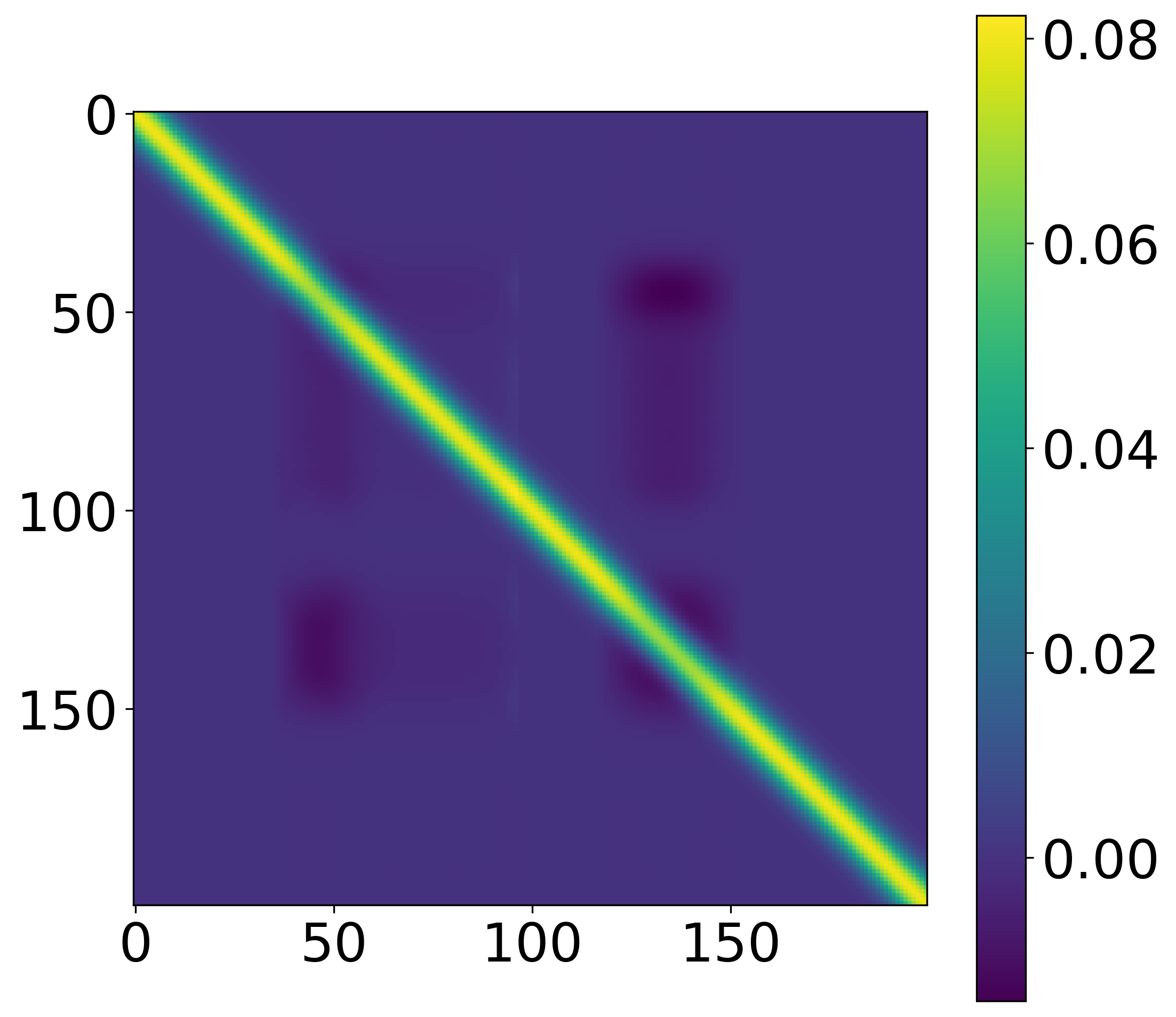}
& \includegraphics[width=0.23\textwidth]{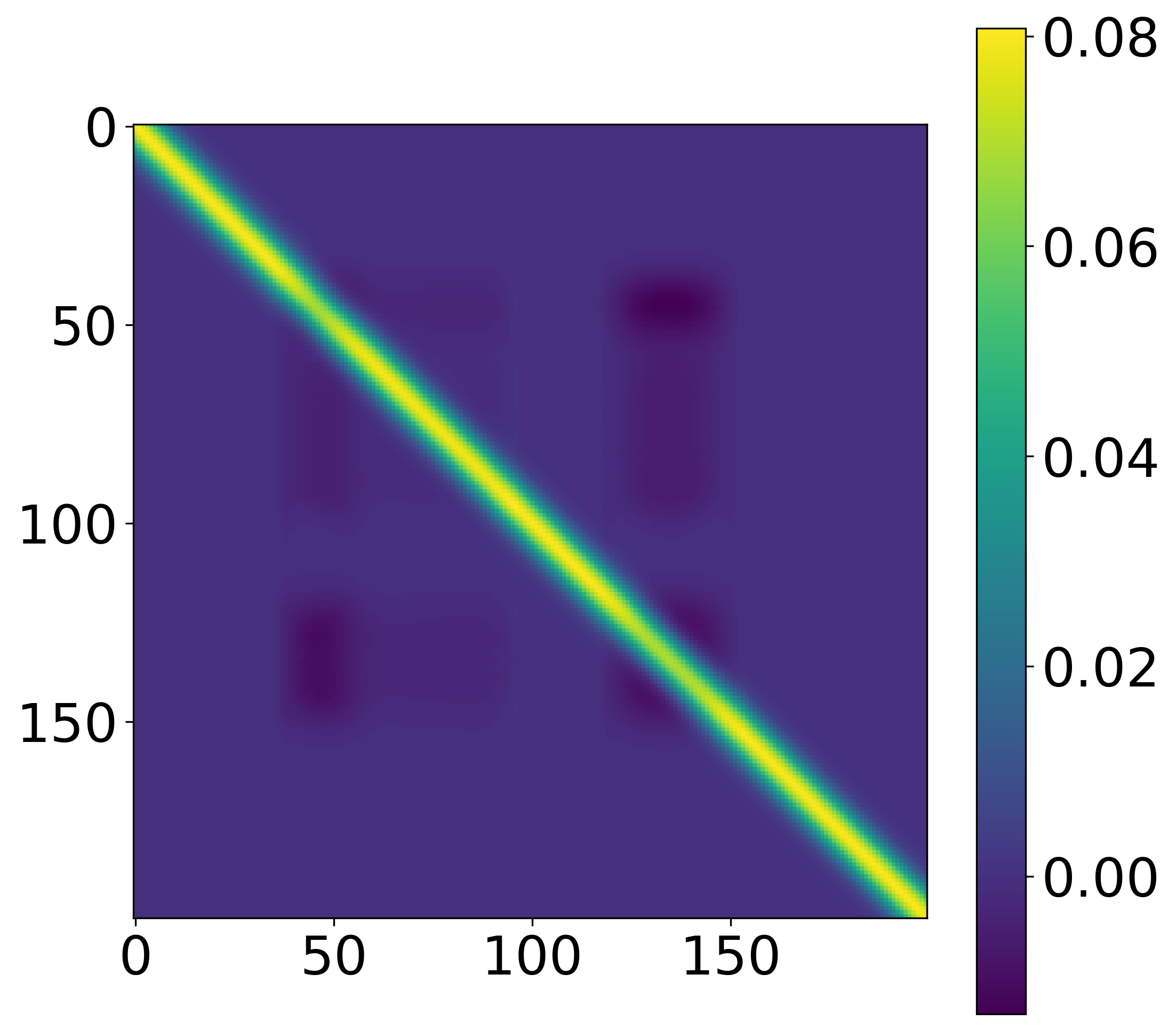}
& \includegraphics[width=0.23\textwidth]{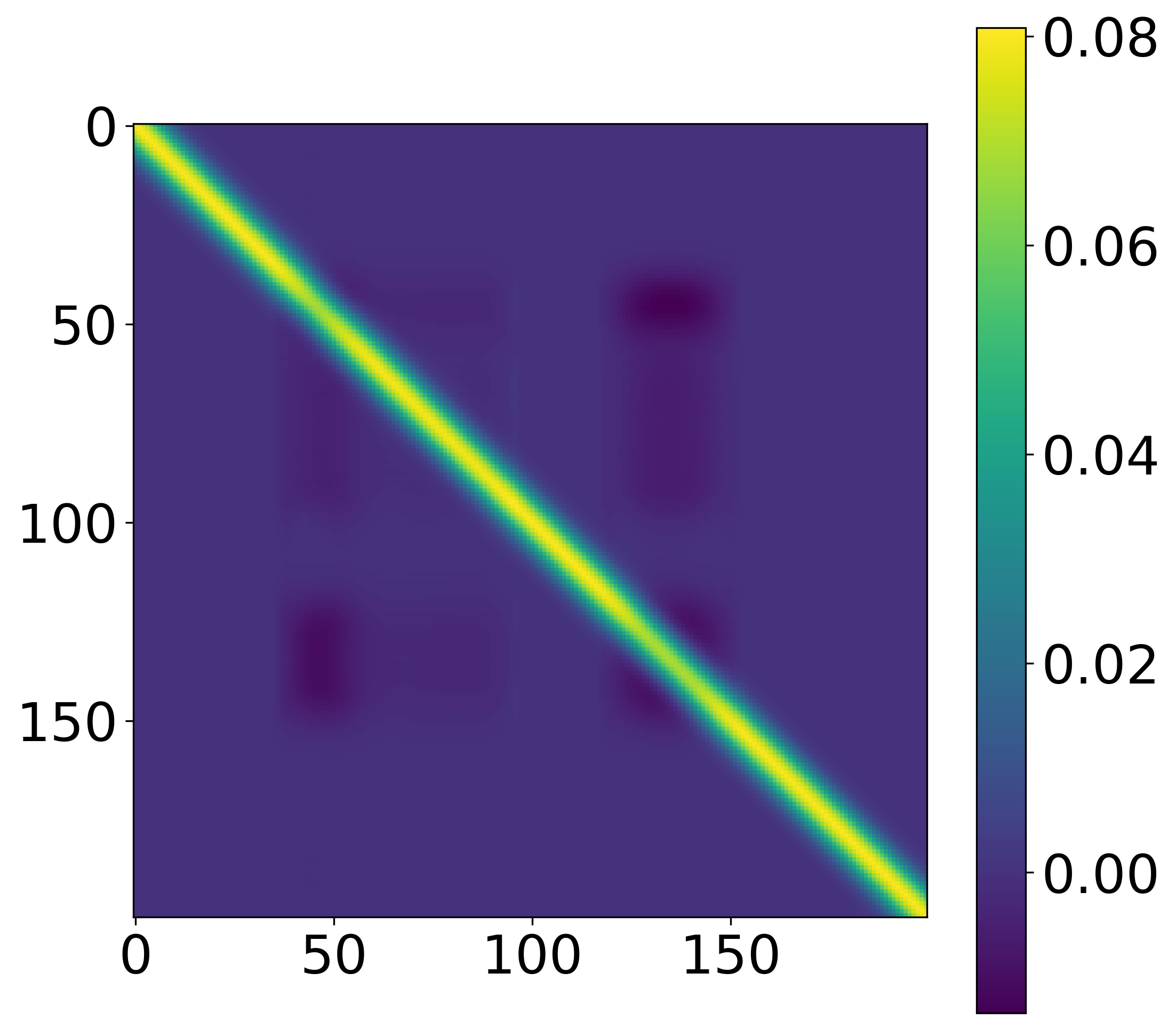} \\
(a) $A=B_0$ & (b) ADP LISTA & (c) ADP IFT & (d) \Our{}
\end{tabular}
\caption{Forward operator in 1D deconvolution. (a) Initial operator $A=B_0$ (Gaussian kernel); $B^*$ with  (b) ADP LISTA, (c) ADP IFT, and (d) \Our{}.}
\label{fig:1D_reconst}
\end{figure}

\begin{figure}[t!]
    \centering
    \begin{tabular}{@{}c@{\qquad}c@{}}
\includegraphics[width=0.35\textwidth]{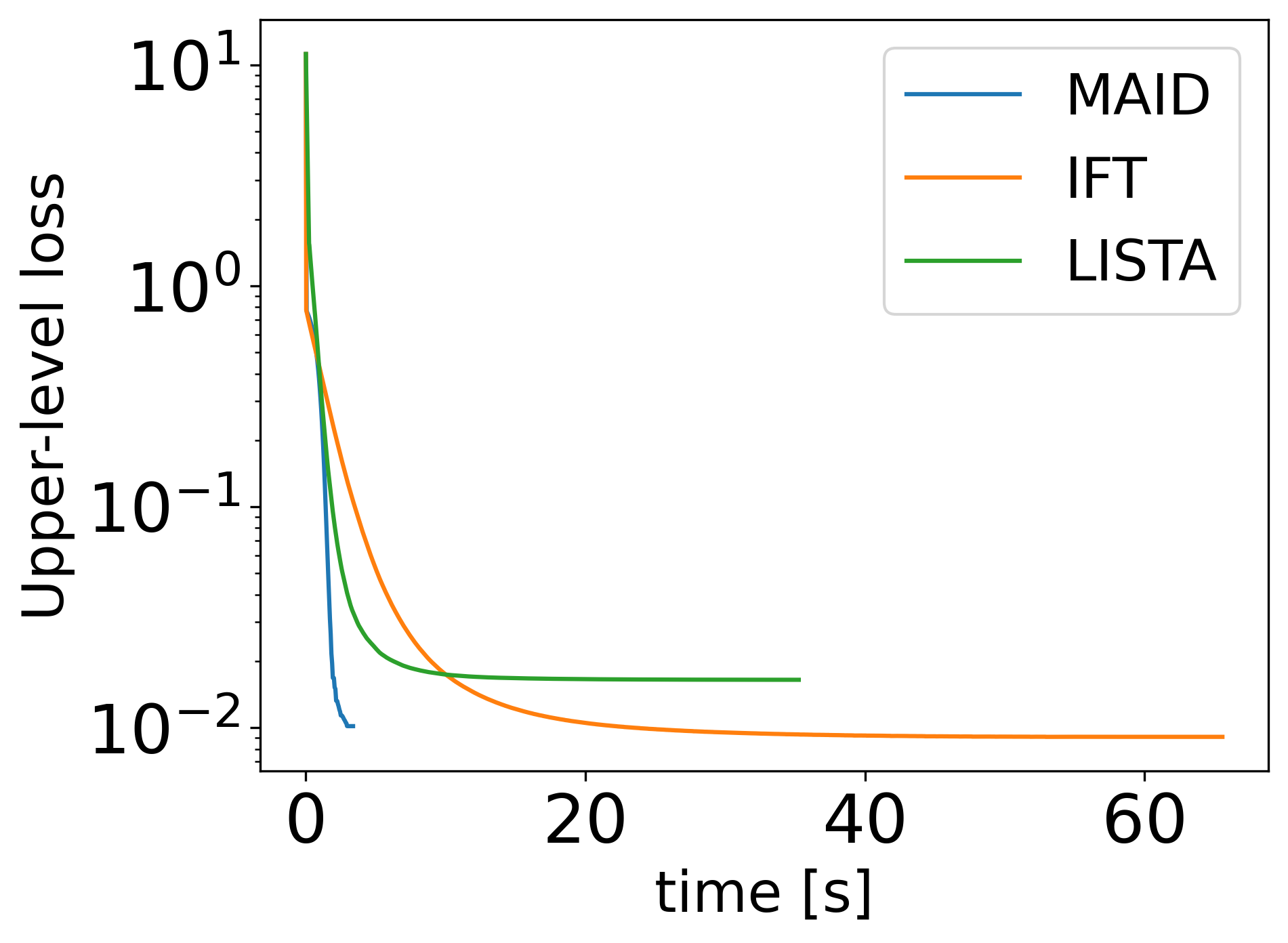}  
& \includegraphics[width=0.35\textwidth]{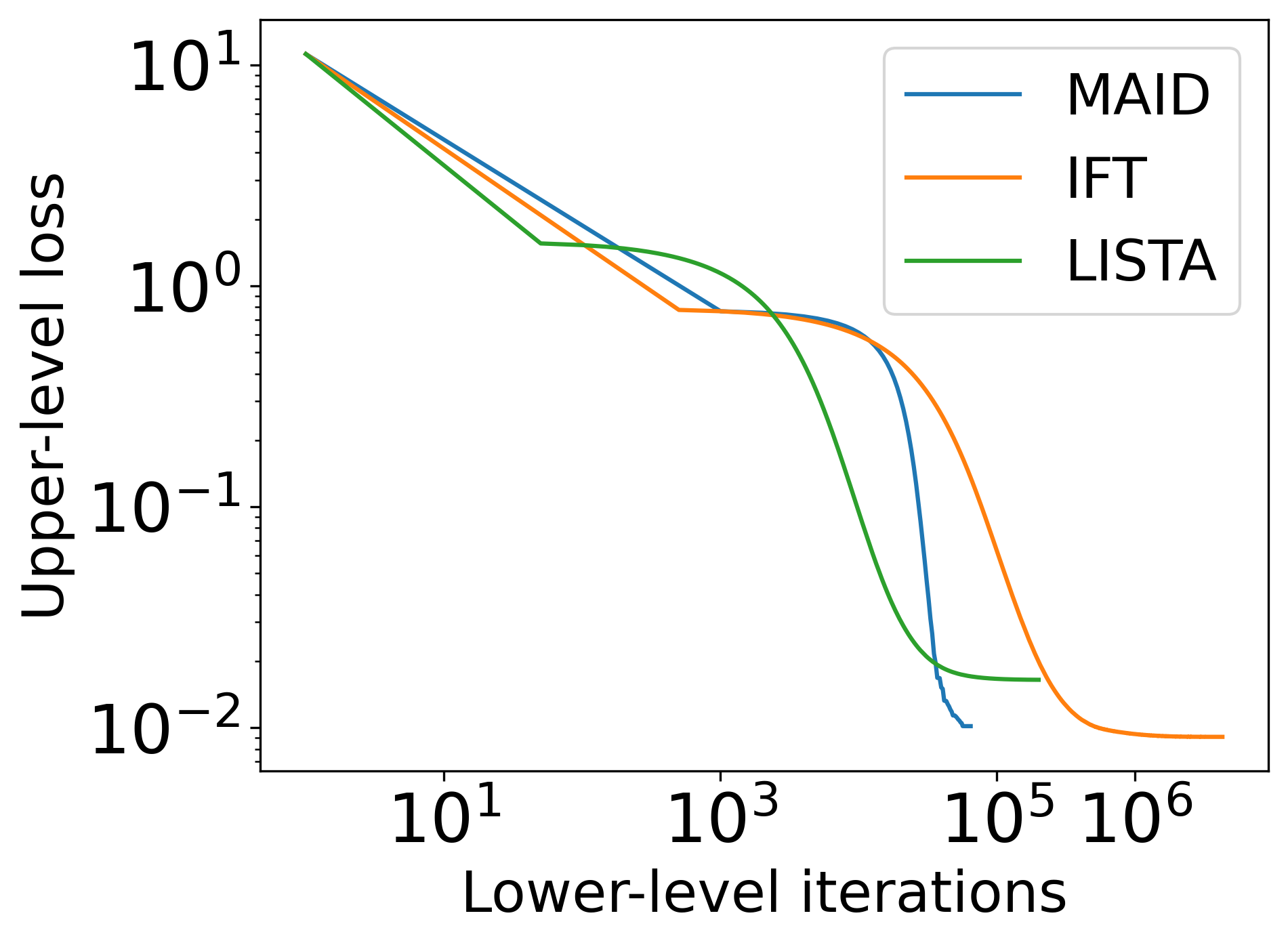}  
\end{tabular}
\caption{Comparison of ADP IFT, ADP LISTA and \Our{} in terms of upper-level loss as a function of wall-clock time and lower-level iterations.}
\label{fig:1D_compare_cost_time}  
\end{figure}

\subsection{2D deconvolution and semi-blind deconvolution}
In this section, we implement {ADP-$\beta$}
with Sobolev regularization for the upper-level (i.e., \eqref{eq:ADP-Sobolev} with $\beta \not=0$), and 2D smoothed TV for the lower-level regularizer, i.e. $\mathcal{J}(x) = \| \nabla x\|_\nu$, where $\|x\|\nu = \sum_{i=1}^N \sqrt{x_i^2 + \nu^2} - \nu$ with $\nu = 10^{-4}$. Given the higher dimensionality of the problem, in the implementation, we take a kernel $b_0$ instead of the entire convolution operator $B_0$ as the initial kernel. 
We consider color images of size $300 \times 400 \times 3$. The regularization parameters for the lower-level and upper-level problems are set as $\alpha = 0.6$ and $\beta = 0.3$, respectively. The value of $\alpha$ and $\beta$ are determined {using grid search}. {To solve ADP-$\beta$, we set the same parameters for \Our{} as in~\cref{sec:1Ddeconvolution}. We use the {Accelerated Gradient Descent Algorithm (AGD)} for smooth functions~\cite[Algorithm 3]{Chambolle_Pock_2016}, to solve the lower-level problem. In step \ref{CG_step} of \cref{Inexact_grad}, we use CG with tolerance $\delta = \epsilon$, where $\epsilon$ denotes the accuracy of the lower-level solution. }

\smallskip 
\textbf{Motion Blur.}
In this case,  $A$ is the 2D convolution operator associated with a motion kernel of size $5\times5$, with non-zero diagonal elements (cf.~$b_0$ in {\cref{fig:kernel_motion}}). The corresponding blurry image is corrupted with additional Gaussian noise with zero mean and standard deviation $\sigma = 0.02$. 

To solve ADP-$\beta$, we set the same parameters for \Our{} as in~\cref{sec:1Ddeconvolution}. We use the Fast Iterative Shrinkage-Thresholding Algorithm (FISTA)  for strongly convex and smooth functions~\cite[Algorithm 5]{Chambolle_Pock_2016}, to solve the lower-level problem. In step \ref{CG_step} of \cref{Inexact_grad}, we use CG with tolerance $\delta = \epsilon$, where $\epsilon$ denotes the accuracy of the lower-level solution. To satisfy the assumptions of~\cref{Inexact_grad}, we use smoothed TV as the lower-level regularizer, i.e., $\reg_\nu(x) = \| \nabla x\|_\nu$, where $\|x\|\nu = \sum_{i=1}^N \sqrt{x_i^2 + \nu^2} - \nu$ with $\nu = 10^{-4}$.

The Sobolev norm $\|B\|_{H^1}^2 = C (\|b\|_{L_2}^2 + \|\nabla_x b\|_{L_2}^2 + \|\nabla_y b\|_{L_2}^2)$ in \eqref{upper_Sobolev}, where $\nabla_x$ and $\nabla_y$ denote the forward difference operators in the horizontal and vertical directions, respectively, and $C > 0$ is some constant absorbed into the regularization parameter,  is computed using the kernel $b$ associated with $B$, that is,  $\|b\|_{H^1}^2 = \|\nabla b\|_2^2$, where $\nabla$ is the forward difference operator. 
The upper-level iterations are stopped  when the accuracy in \Our{} reaches $\epsilon = 10^{-6}$. 

As illustrated in \cref{fig:all}, the deblurred images obtained by solving the ADP-$\beta$ bilevel problem using \Our{} are superior, both visually and in terms of quality metrics, to the results obtained by solving only the lower-level variational problem with the initial kernel $b_0$. PSNR and SSIM values for all reconstructions are reported on top of each image in \cref{fig:all}. By way of example, we report in {\cref{fig:kernel_motion}} the first channel of the kernel $b^*$ obtained by solving the ADP-$\beta$ problem, along with the absolute difference $|b^* - b_0|$, corresponding to the reconstructions in the first row of \cref{fig:all}. These results indicate off-diagonal changes in the initial kernel. 

\begin{figure}[h!]
\centering
\begin{tabular}{@{}c@{\quad}c@{\quad}c@{}}
\includegraphics[width=0.3\textwidth]{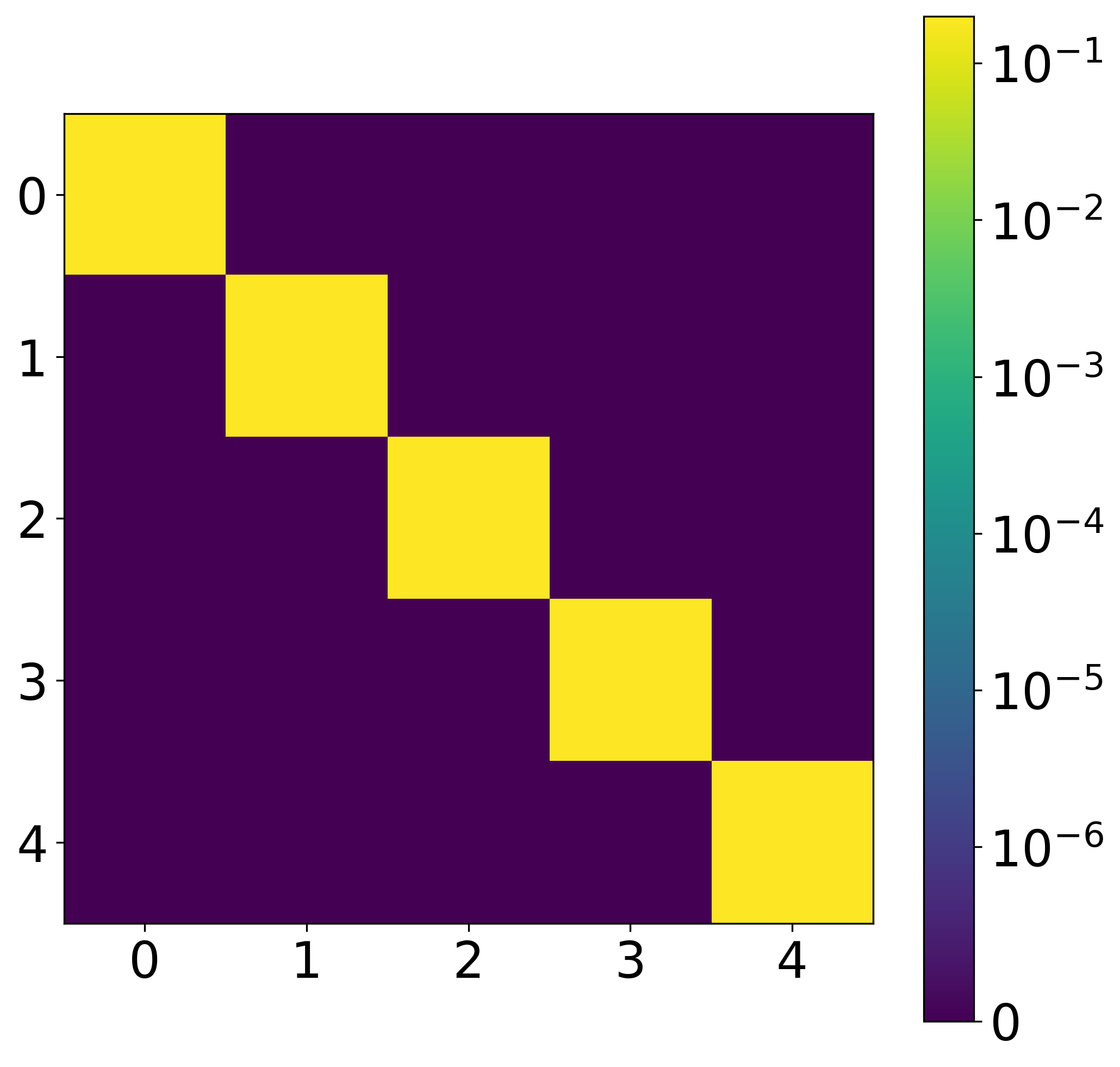}     
&  \includegraphics[width=0.3\textwidth]{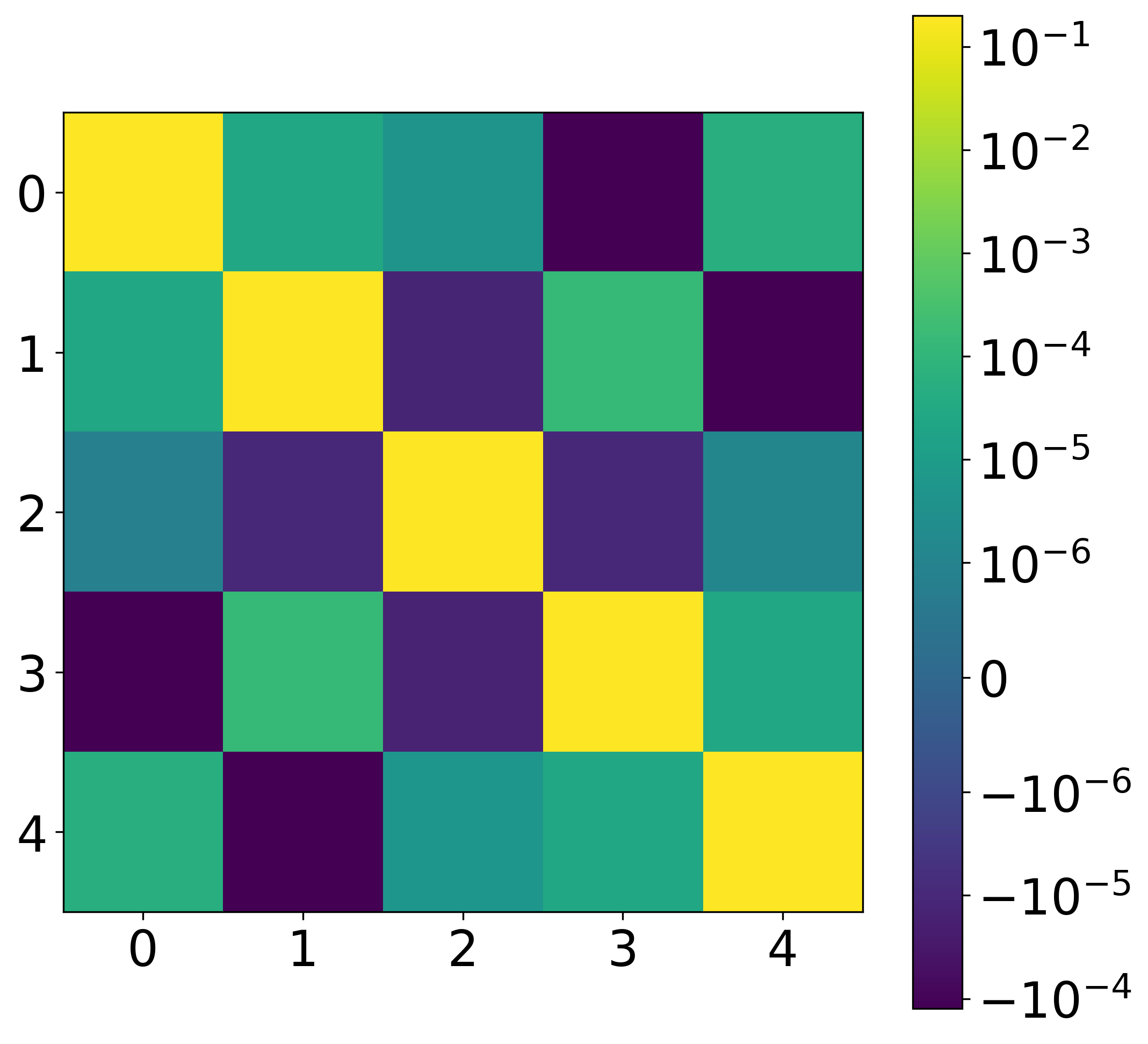}
& \includegraphics[width=0.3\textwidth]{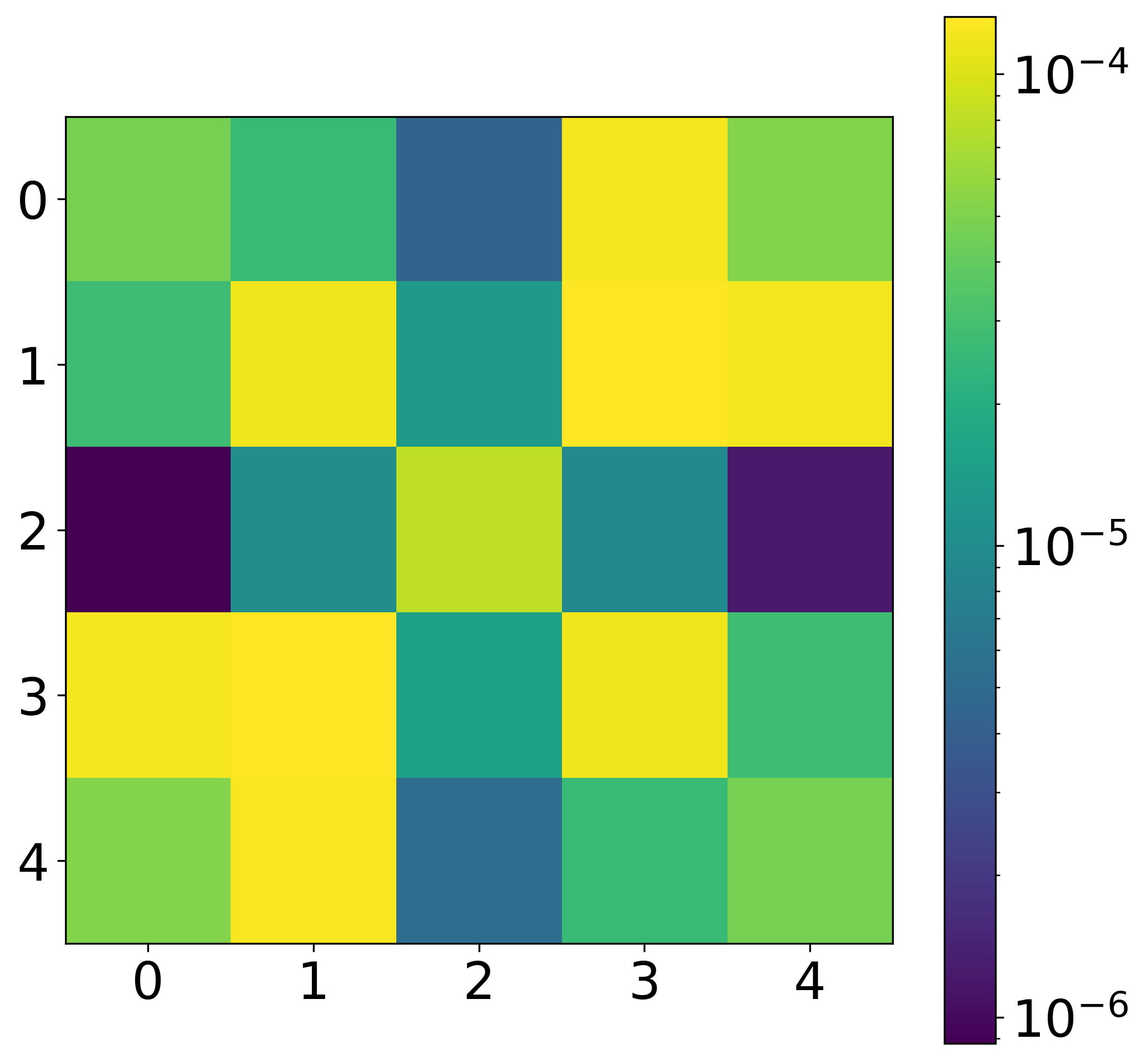} \\
\footnotesize{(a) Initial kernel $b_0$}  & 
\footnotesize{(b) Kernel $b^*$} &
\footnotesize{(c) $| b^* - b_0|$}
\end{tabular}
\caption{Kernels (first channel) of the 2D motion blur forward operator. (a) Initial kernel $b_0$. (b) Optimal kernel $b^*$ recovered by \Our{}. (c) Difference $|b^*-b_0|$.}
    \label{fig:kernel_motion}
\end{figure}

\begin{figure}[t!]
  \centering
  \begin{subfigure}{0.24\textwidth}
    \includegraphics[width=\linewidth]{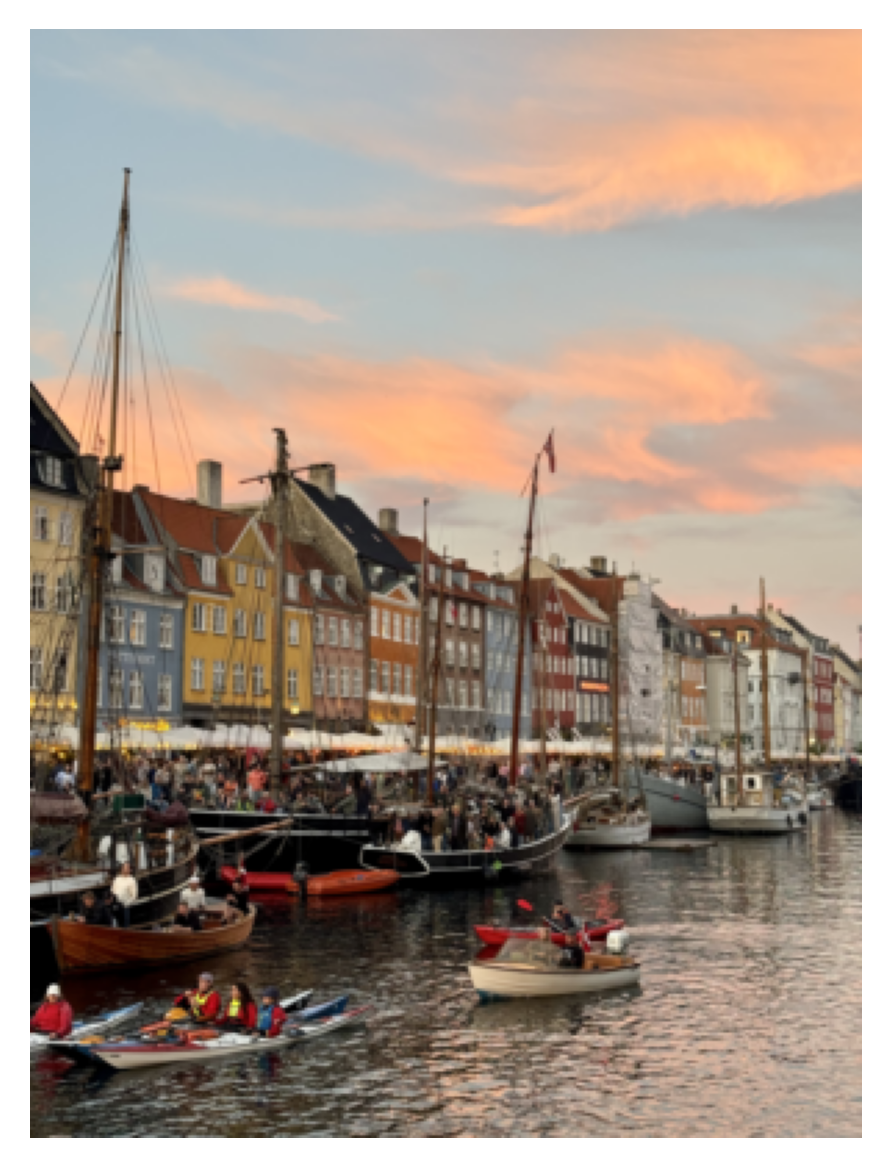}
  \end{subfigure}
  \hfill
  \begin{subfigure}{0.24\textwidth}
    \includegraphics[width=\linewidth]{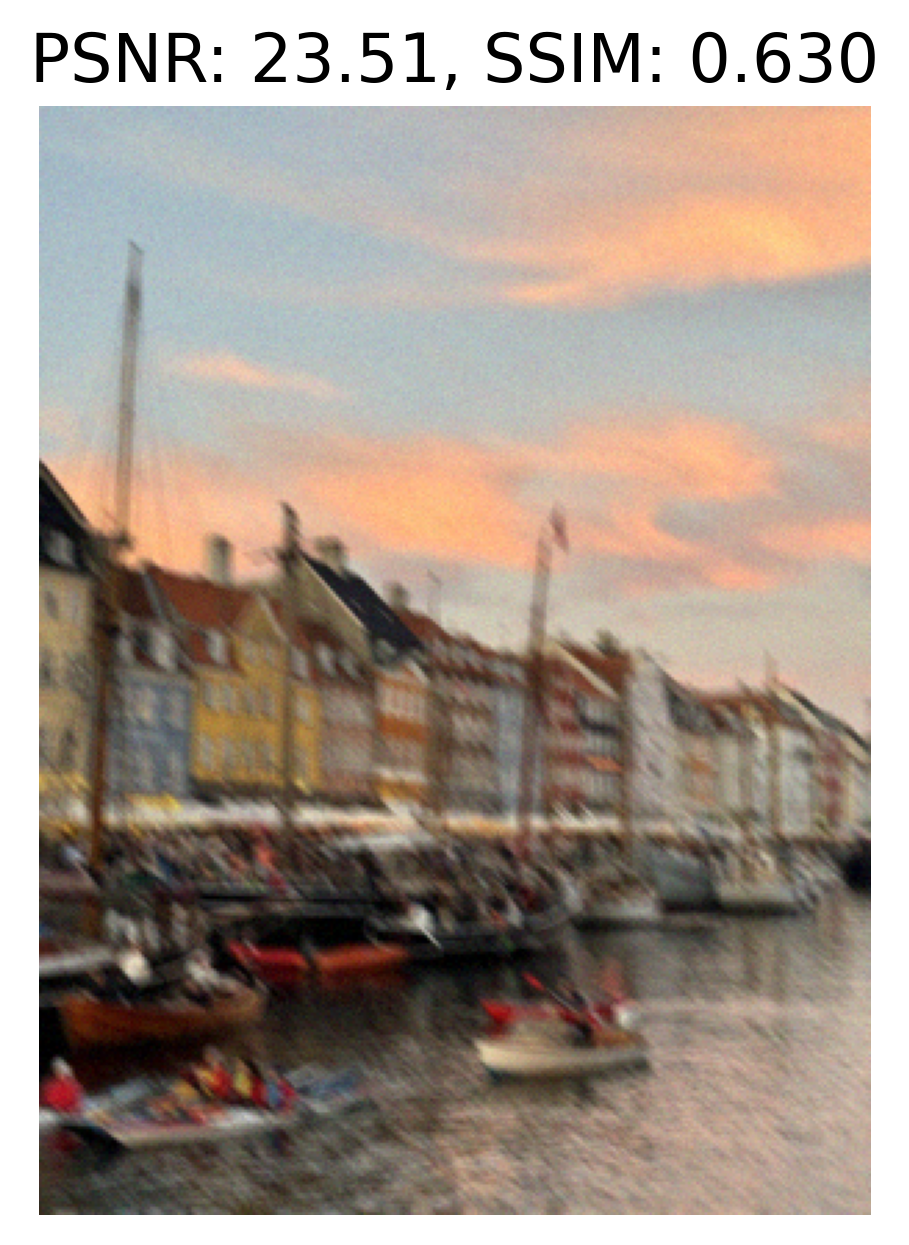}
  \end{subfigure}
  \hfill
  \begin{subfigure}{0.24\textwidth}
    \includegraphics[width=\linewidth]{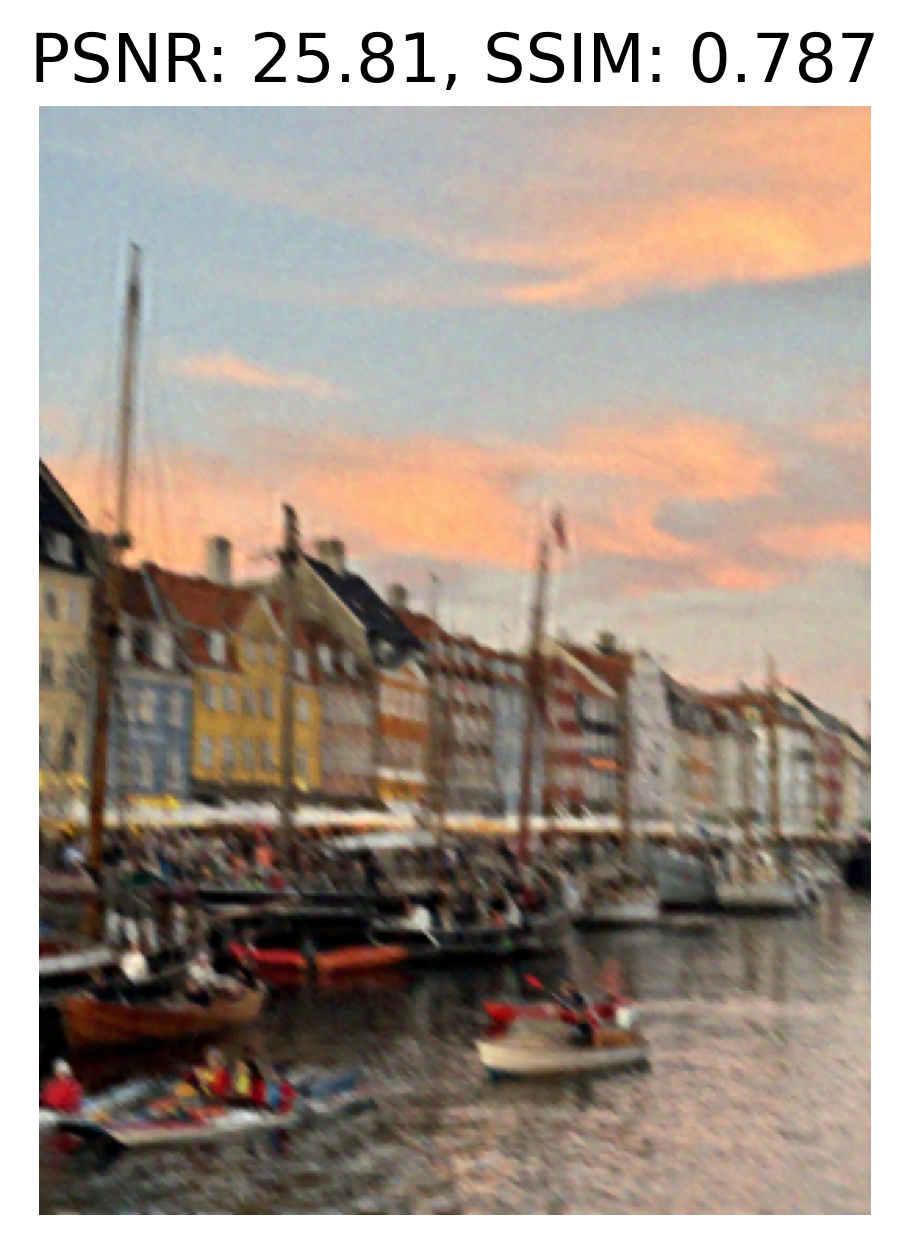}
  \end{subfigure}
  \hfill
  \begin{subfigure}{0.24\textwidth}
    \includegraphics[width=\linewidth]{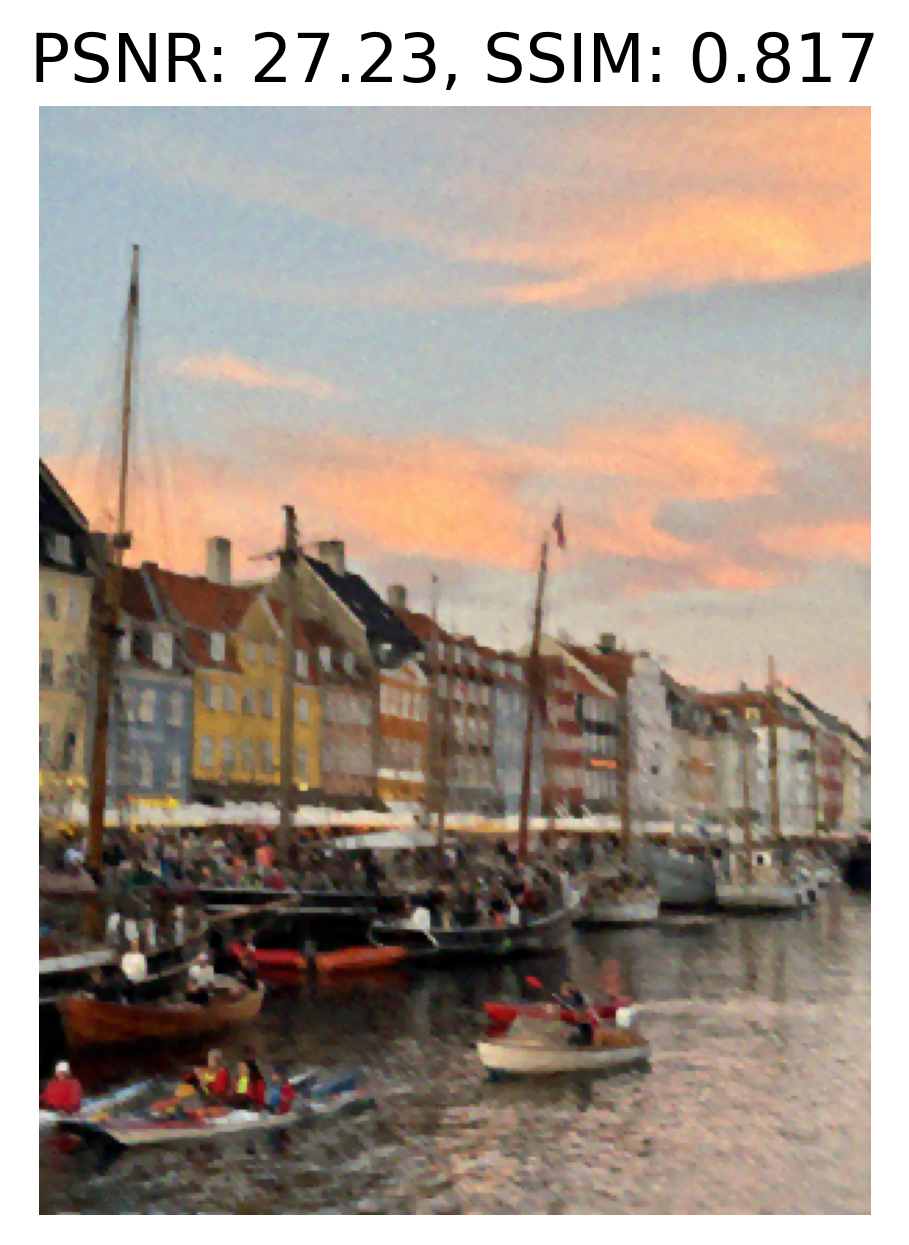}
  \end{subfigure}
  \vspace{1pt}  

  \begin{subfigure}{0.24\textwidth}
    \includegraphics[width=\linewidth]{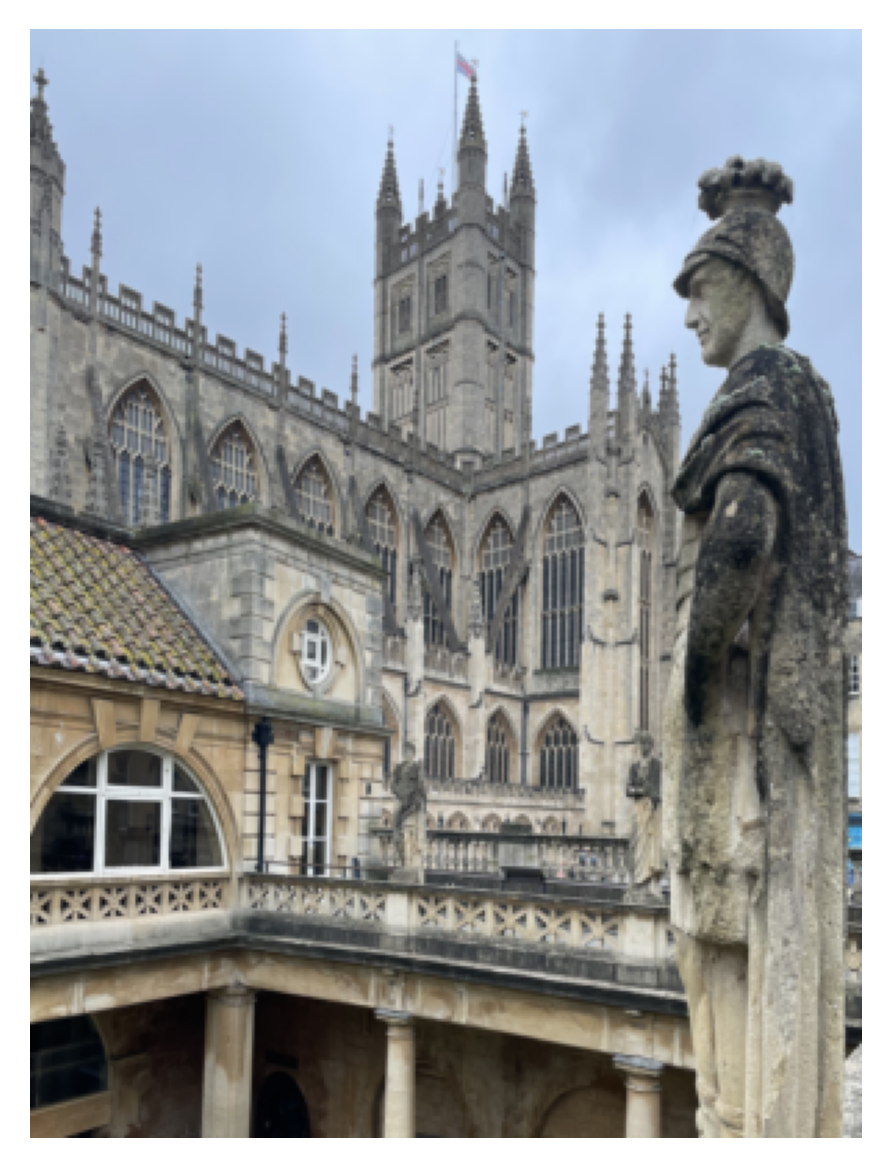}
  \end{subfigure}
  \hfill
  \begin{subfigure}{0.24\textwidth}
    \includegraphics[width=\linewidth]{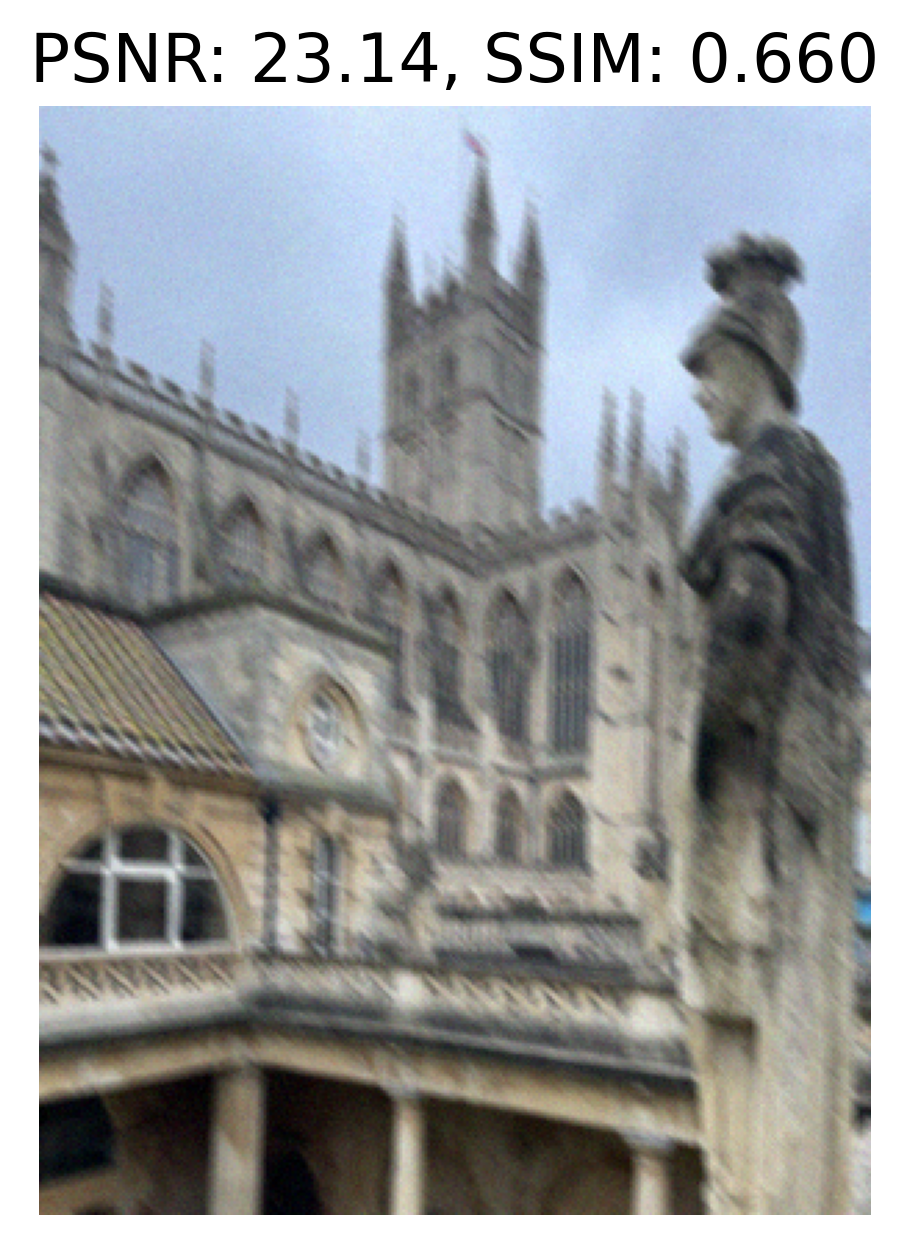}
  \end{subfigure}
  \hfill
  \begin{subfigure}{0.24\textwidth}
    \includegraphics[width=\linewidth]{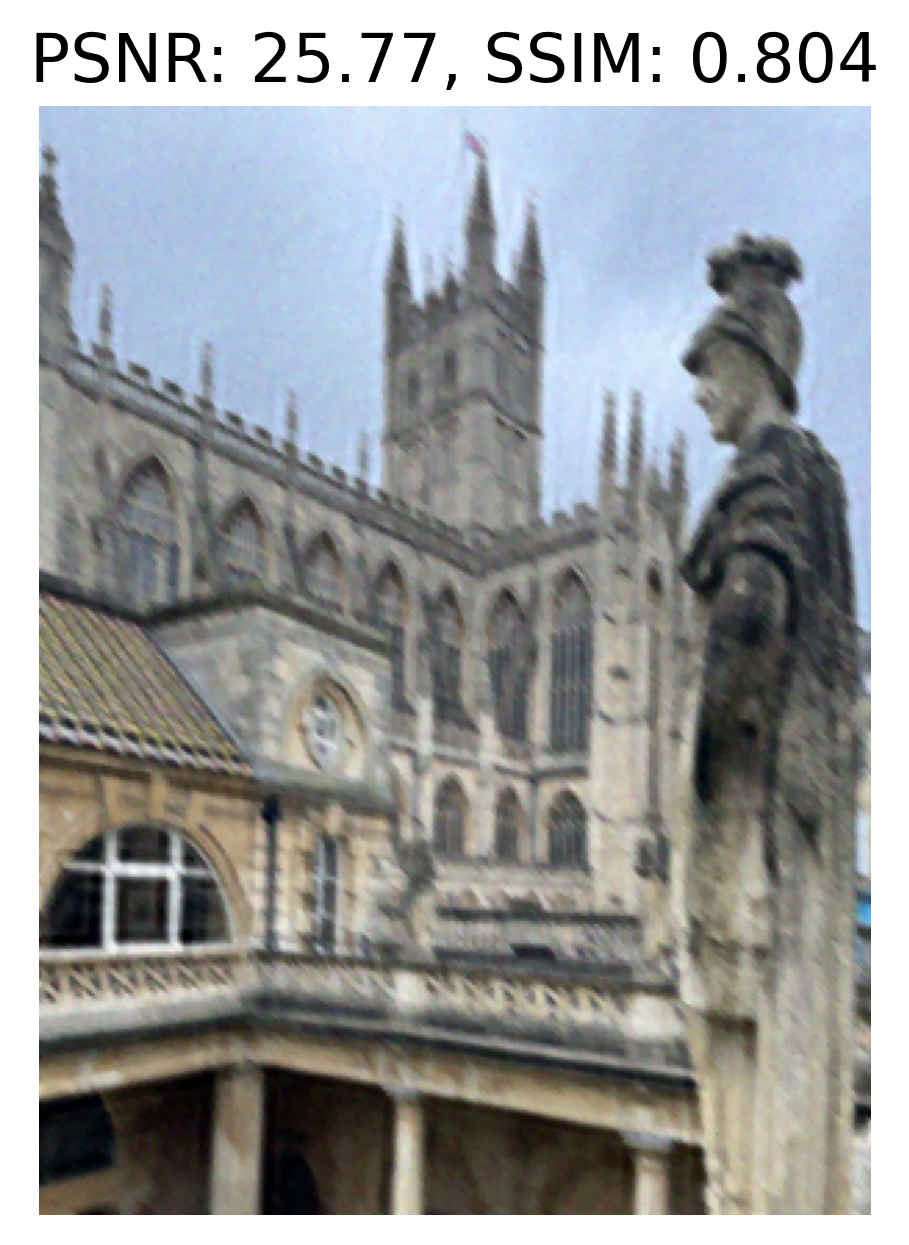}
  \end{subfigure}
  \hfill
  \begin{subfigure}{0.24\textwidth}
    \includegraphics[width=\linewidth]{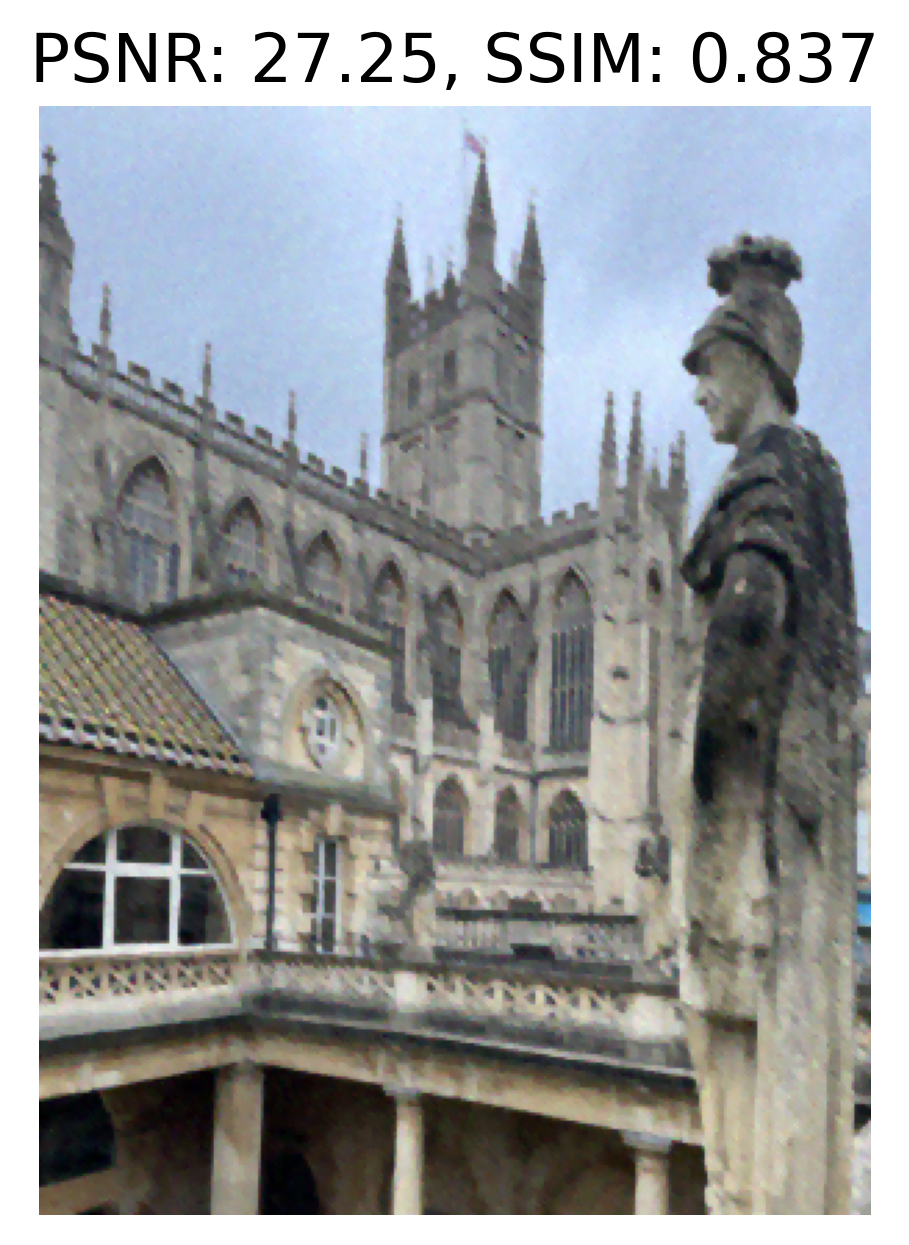}
  \end{subfigure}
  \vspace{1pt}  
  \begin{subfigure}{0.24\textwidth}
    \includegraphics[width=\linewidth]{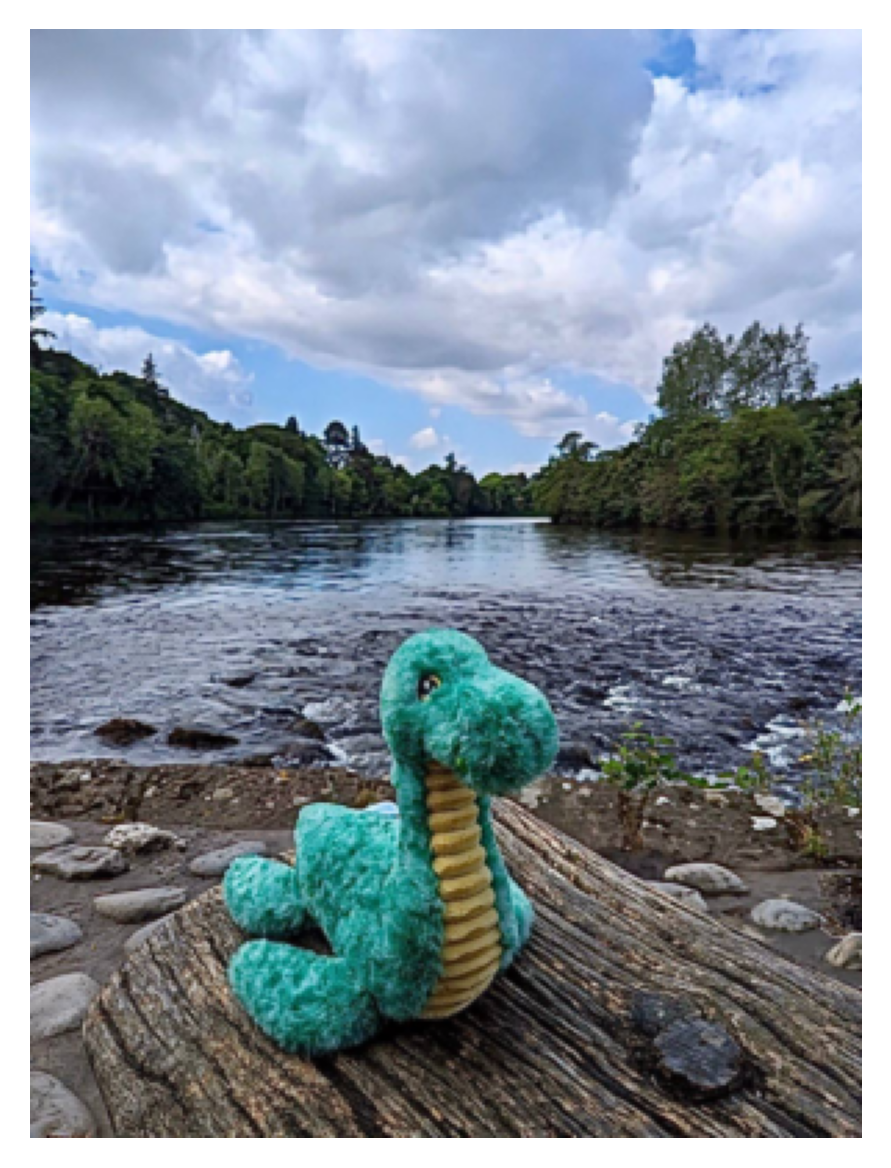}
    \caption{Ground truth}
    \label{fig:image_m}
  \end{subfigure}
  \hfill
  \begin{subfigure}{0.24\textwidth}
        \includegraphics[width=\linewidth]{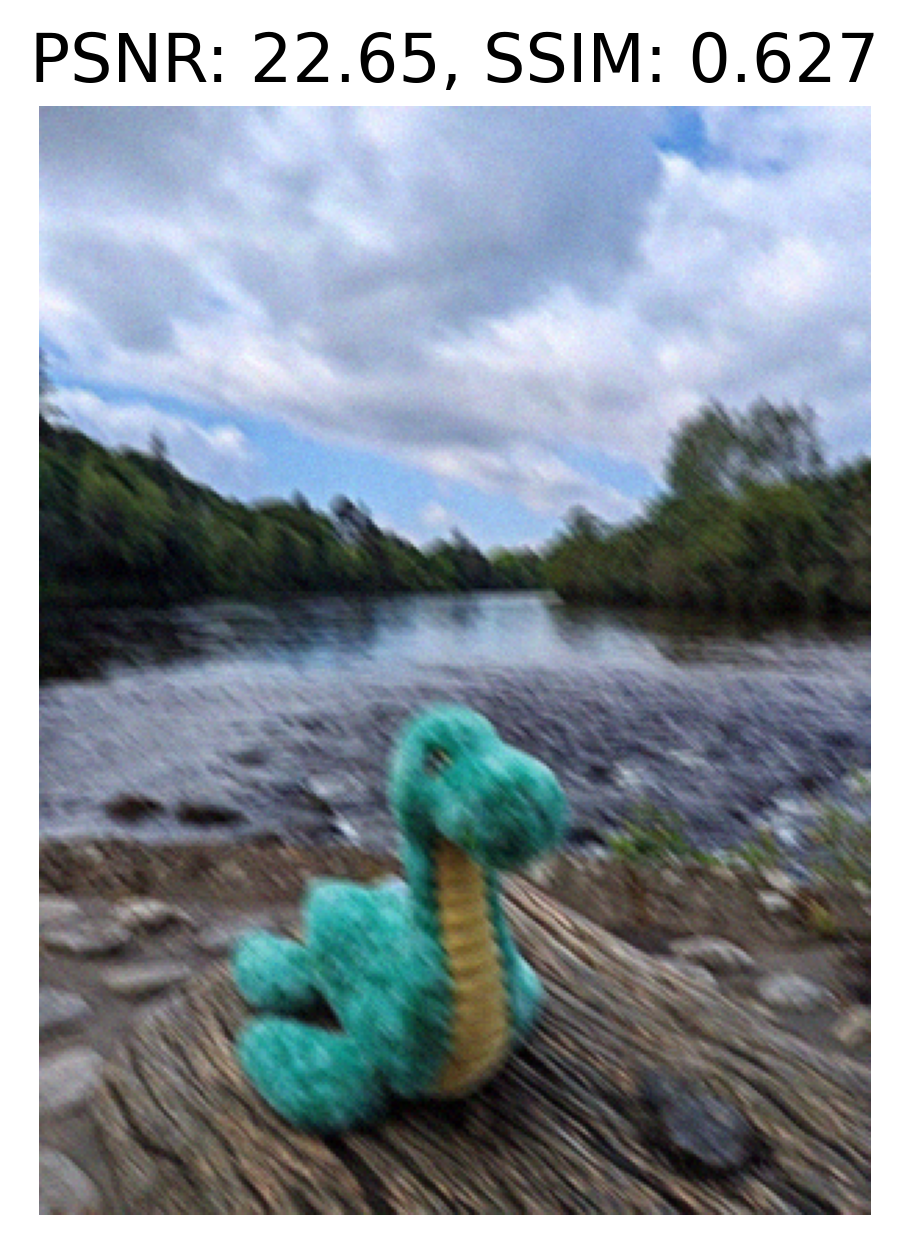}
    \caption{Blurry and noisy}
    \label{fig:blur_m}
  \end{subfigure}
  \hfill
  \begin{subfigure}{0.24\textwidth}
    \includegraphics[width=\linewidth]{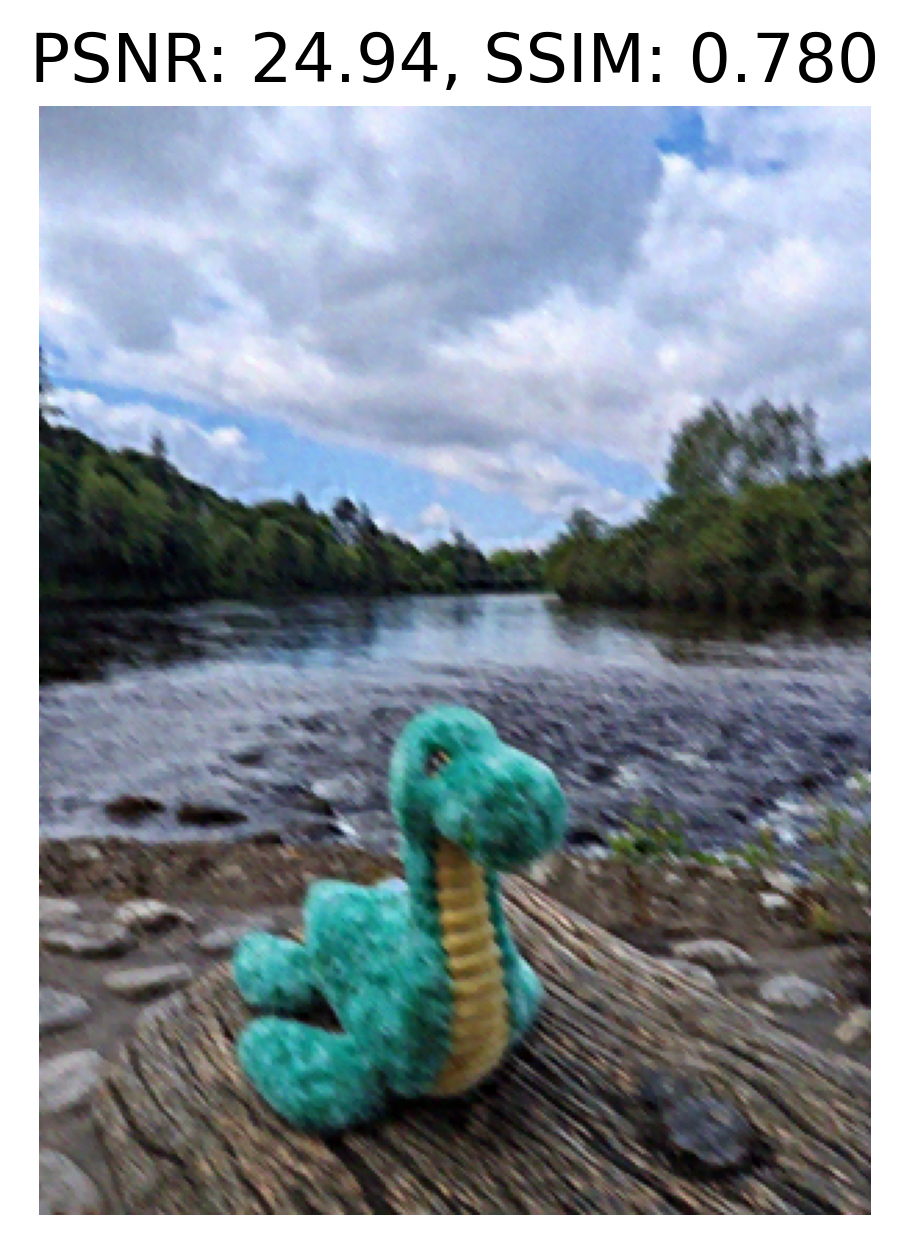}
    \caption{TV}
    \label{fig:TV_m}
  \end{subfigure}
  \hfill
  \begin{subfigure}{0.24\textwidth}
    \includegraphics[width=\linewidth]{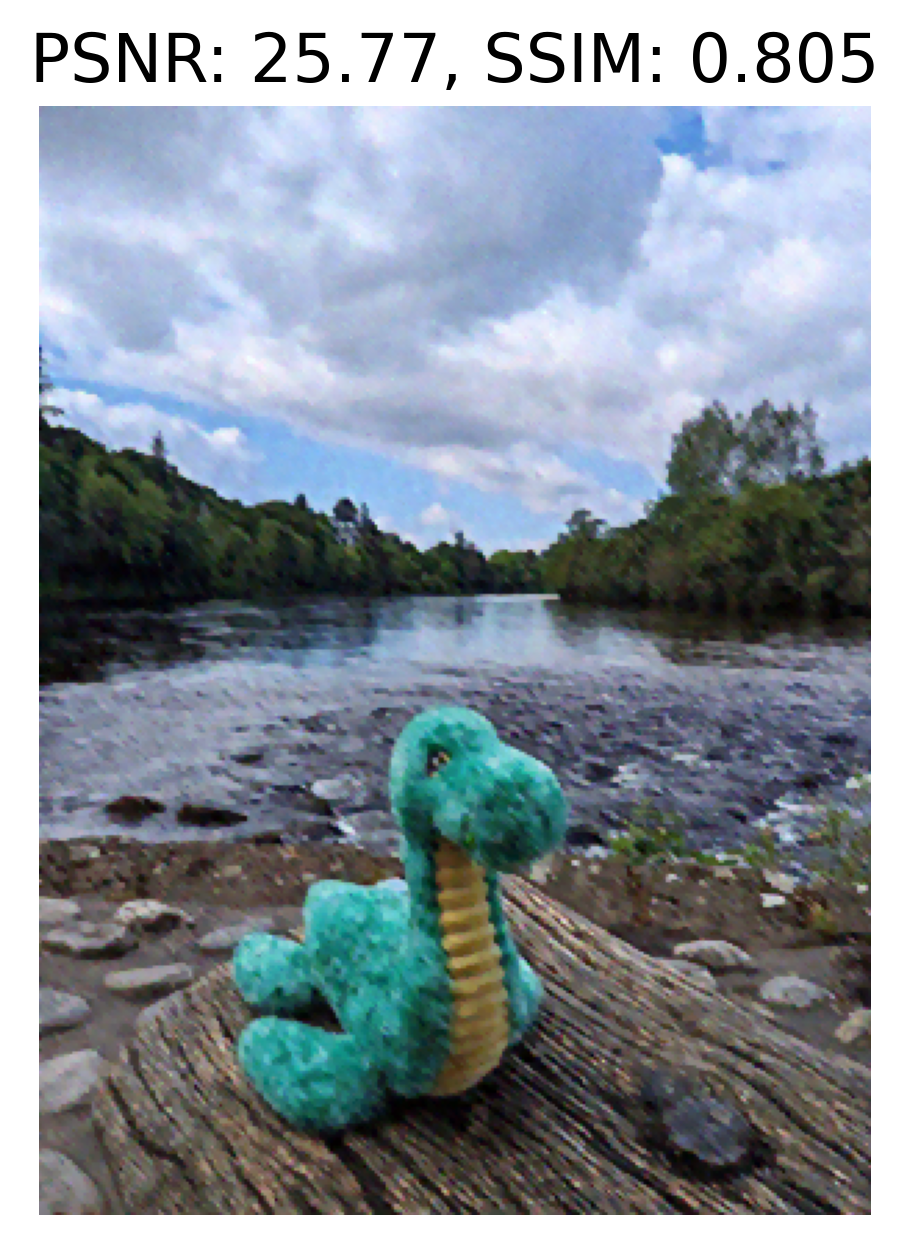}
    \caption{\Our{}}
    \label{fig:ADP_m}
  \end{subfigure}
  \caption{Reconstructions in deblurring with a motion blur 
  }
  \label{fig:all}
\end{figure}

\smallskip 

\textbf{Semi-blind deblurring.}
In this case, we consider two different blur operators, $A_1$ and $A_2$, with corresponding kernels $\widehat{b}_1$ and $\widehat{b}_2$, representing motion blur and Gaussian blur, respectively. The diagonal motion kernel is defined as in the previous section, while for the Gaussian blur, we use a $5 \times 5$ kernel with $\widetilde{\sigma} = 0.5$. Since only the motion blur is assumed to be known, we set $b_0$ as the motion blur kernel $\widehat{b}_1$, as shown in {\cref{fig:kernel_semiblind}}. Moreover, additive Gaussian noise with $\sigma = 0.02$ is added to obtain the blurry and noisy image.
\begin{figure}[t!]
  \centering
  \begin{subfigure}[t]{0.24\textwidth}
    \includegraphics[width=\linewidth]{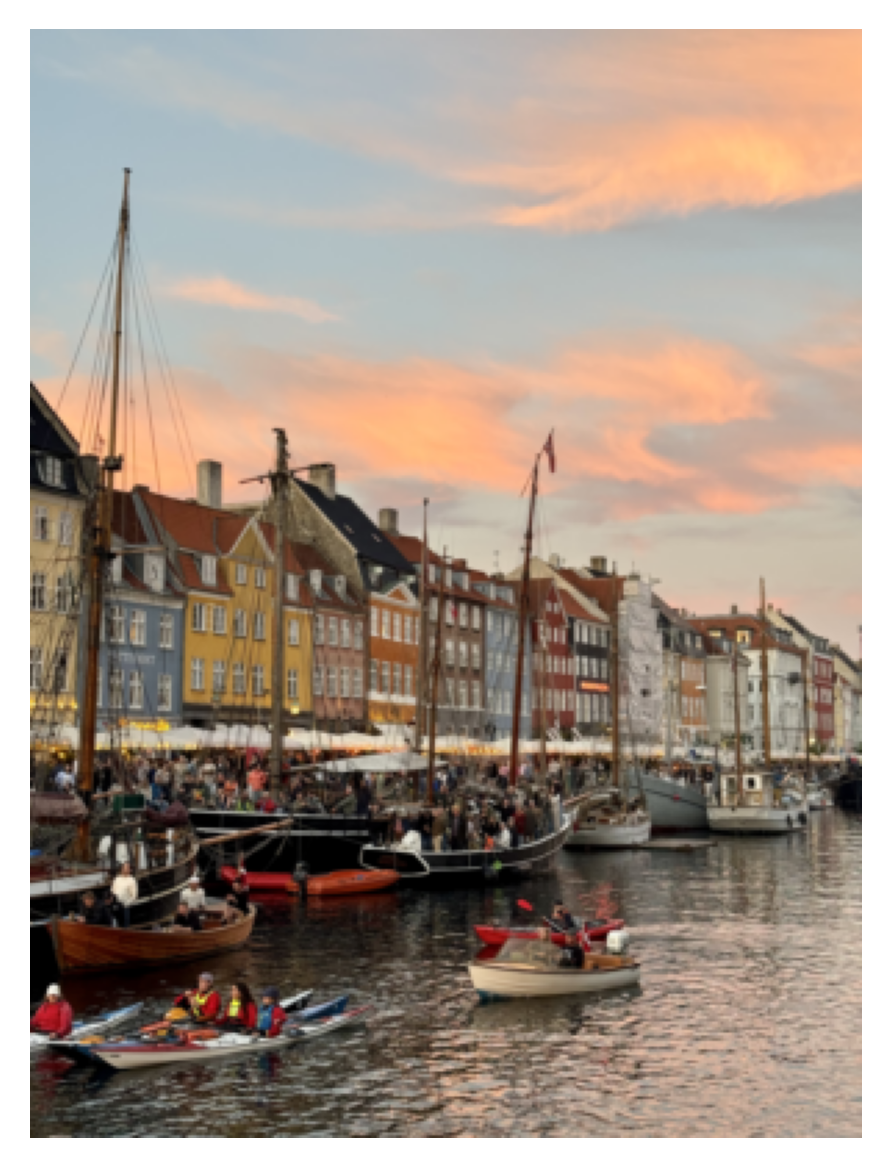}
  \end{subfigure}
  \hfill
  \begin{subfigure}[t]{0.24\textwidth}
    \includegraphics[width=\linewidth]{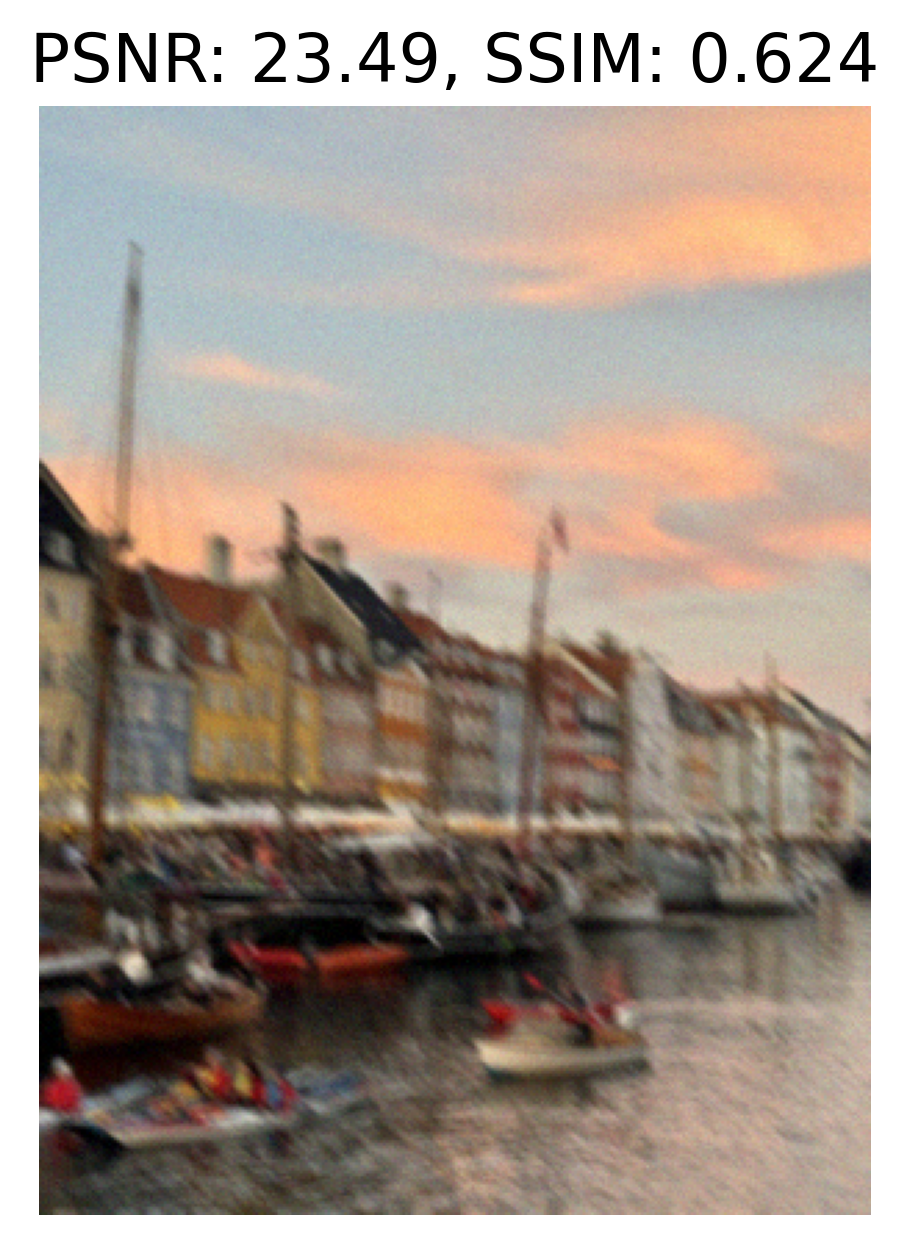}
  \end{subfigure}
  \hfill
  \begin{subfigure}[t]{0.24\textwidth}
    \includegraphics[width=\linewidth]{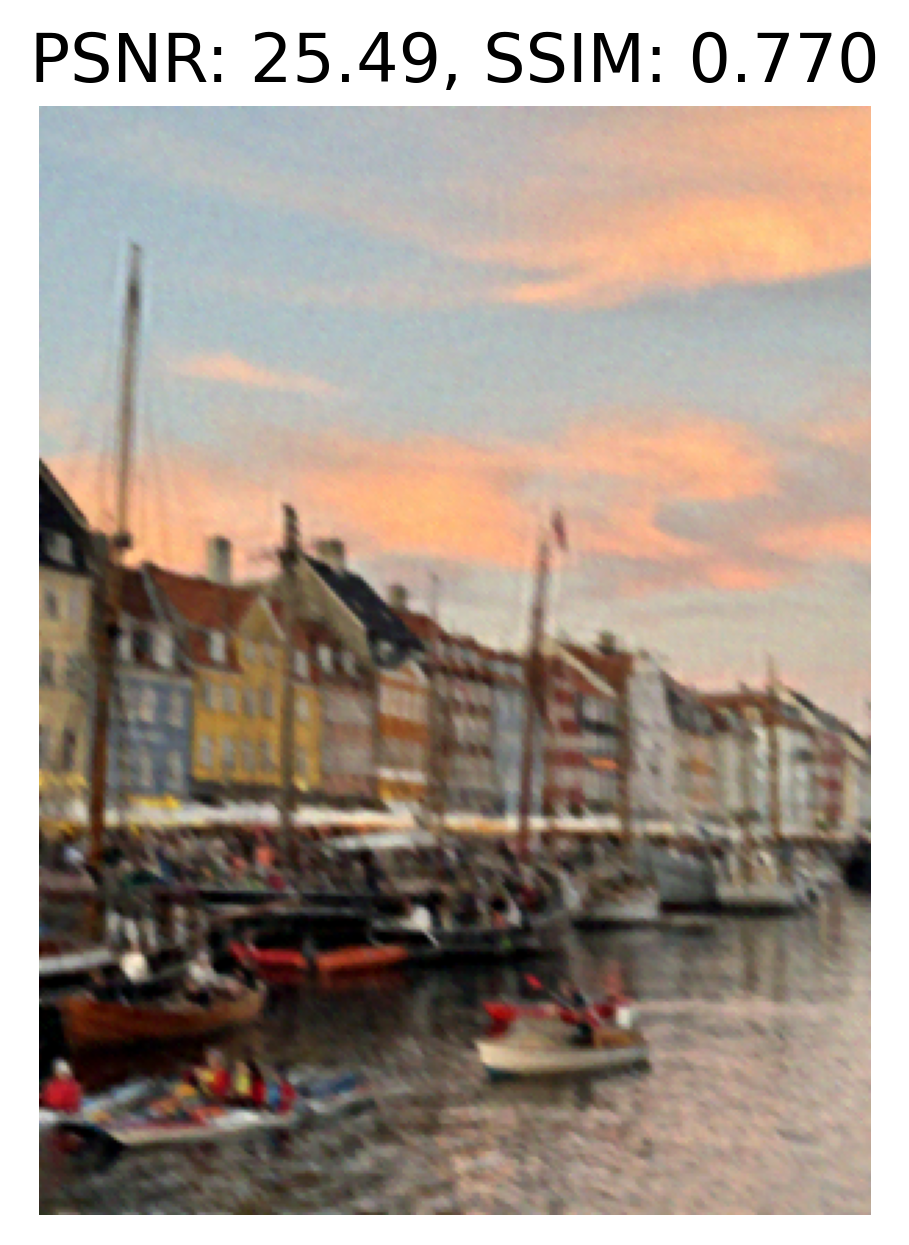}
  \end{subfigure}
  \hfill
  \begin{subfigure}[t]{0.24\textwidth}
    \includegraphics[width=\linewidth]{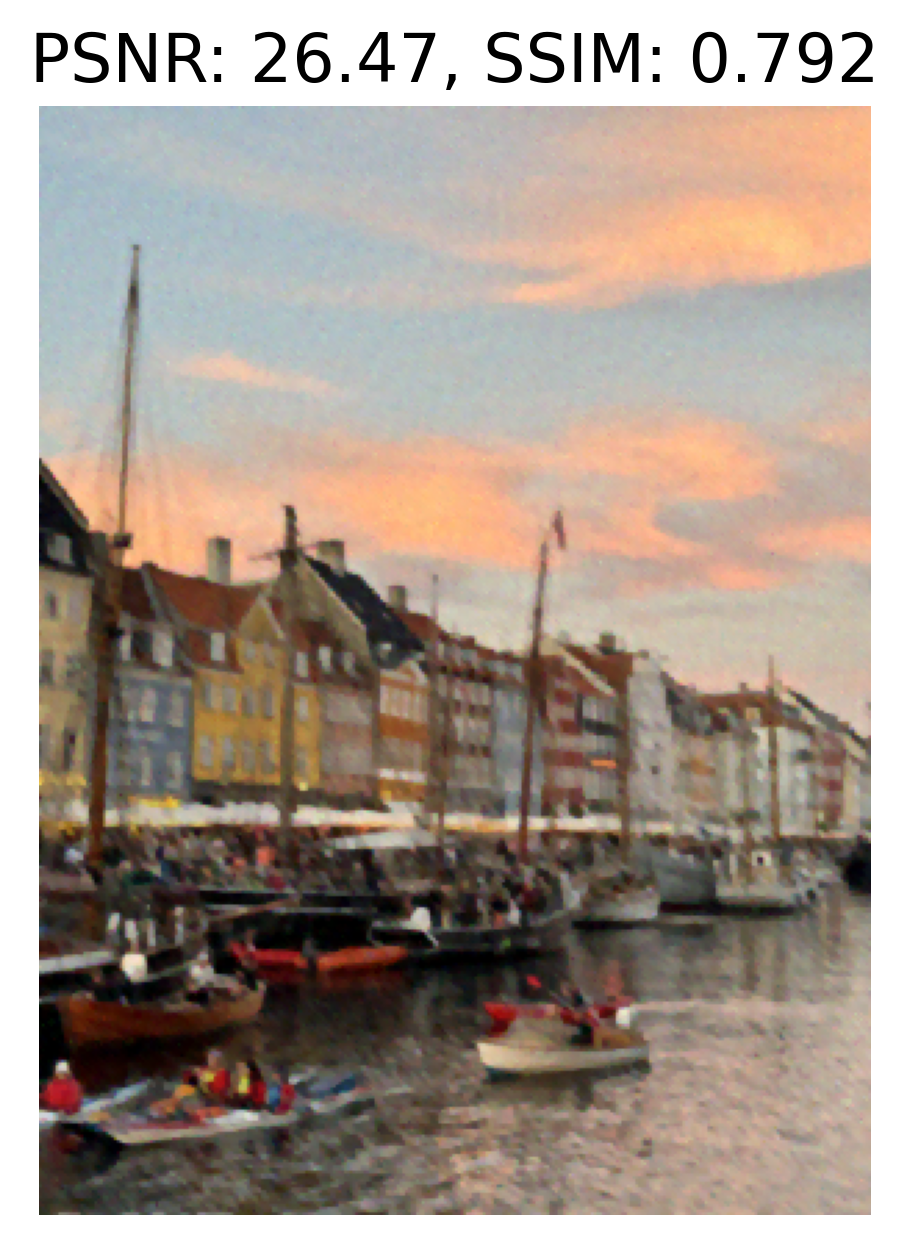}
  \end{subfigure}
  \vspace{1pt}  

  \begin{subfigure}[t]{0.24\textwidth}
    \includegraphics[width=\linewidth]{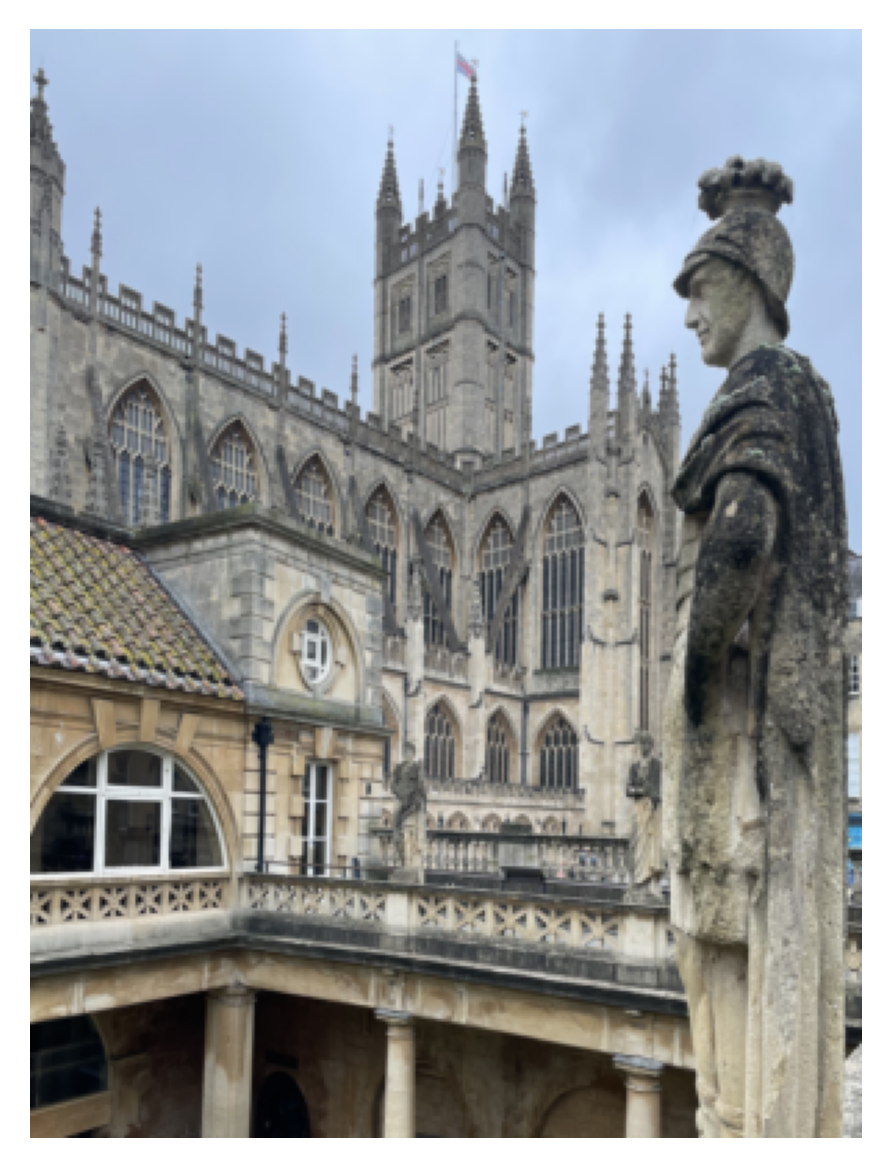}
  \end{subfigure}
  \hfill
  \begin{subfigure}[t]{0.24\textwidth}
    \includegraphics[width=\linewidth]{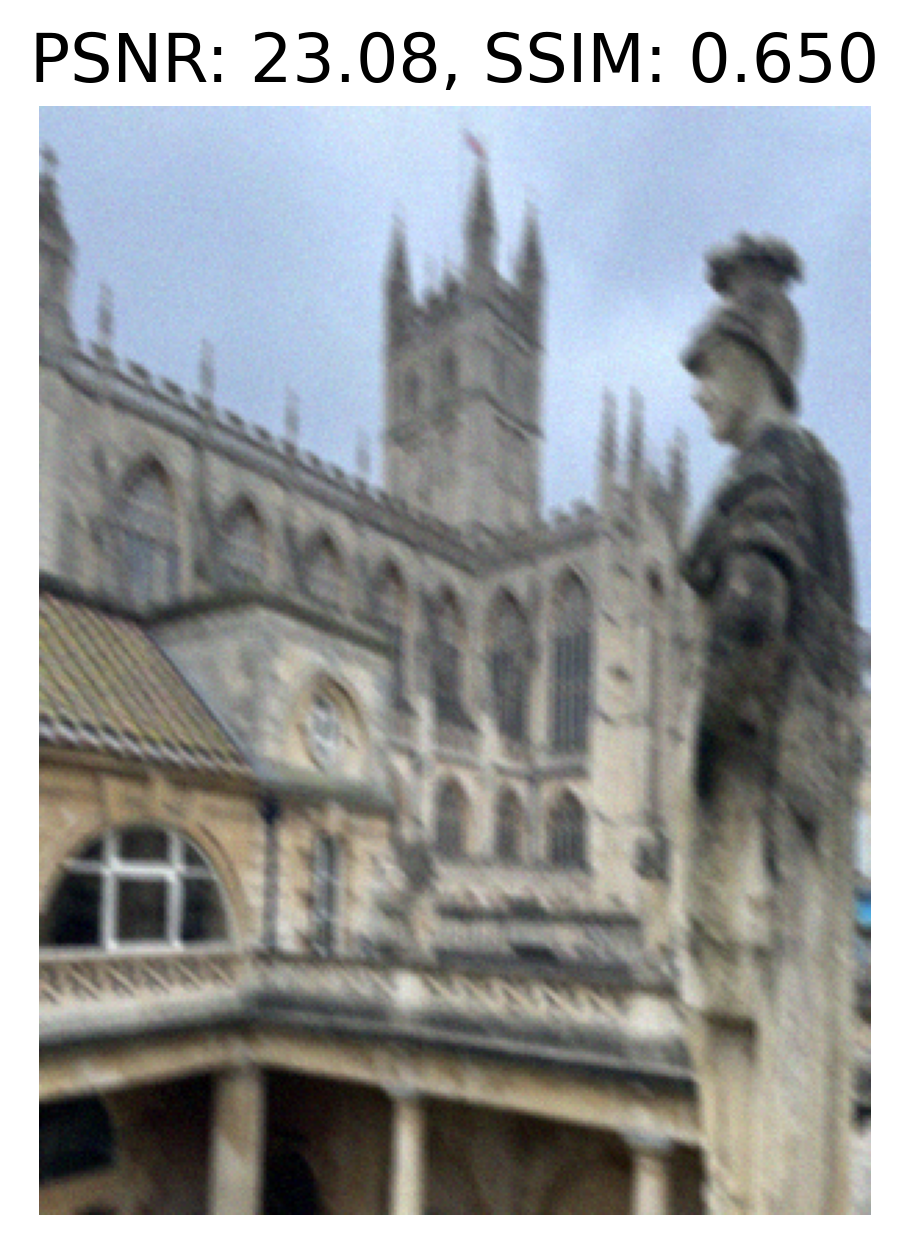}
  \end{subfigure}
  \hfill
  \begin{subfigure}[t]{0.24\textwidth}
    \includegraphics[width=\linewidth]{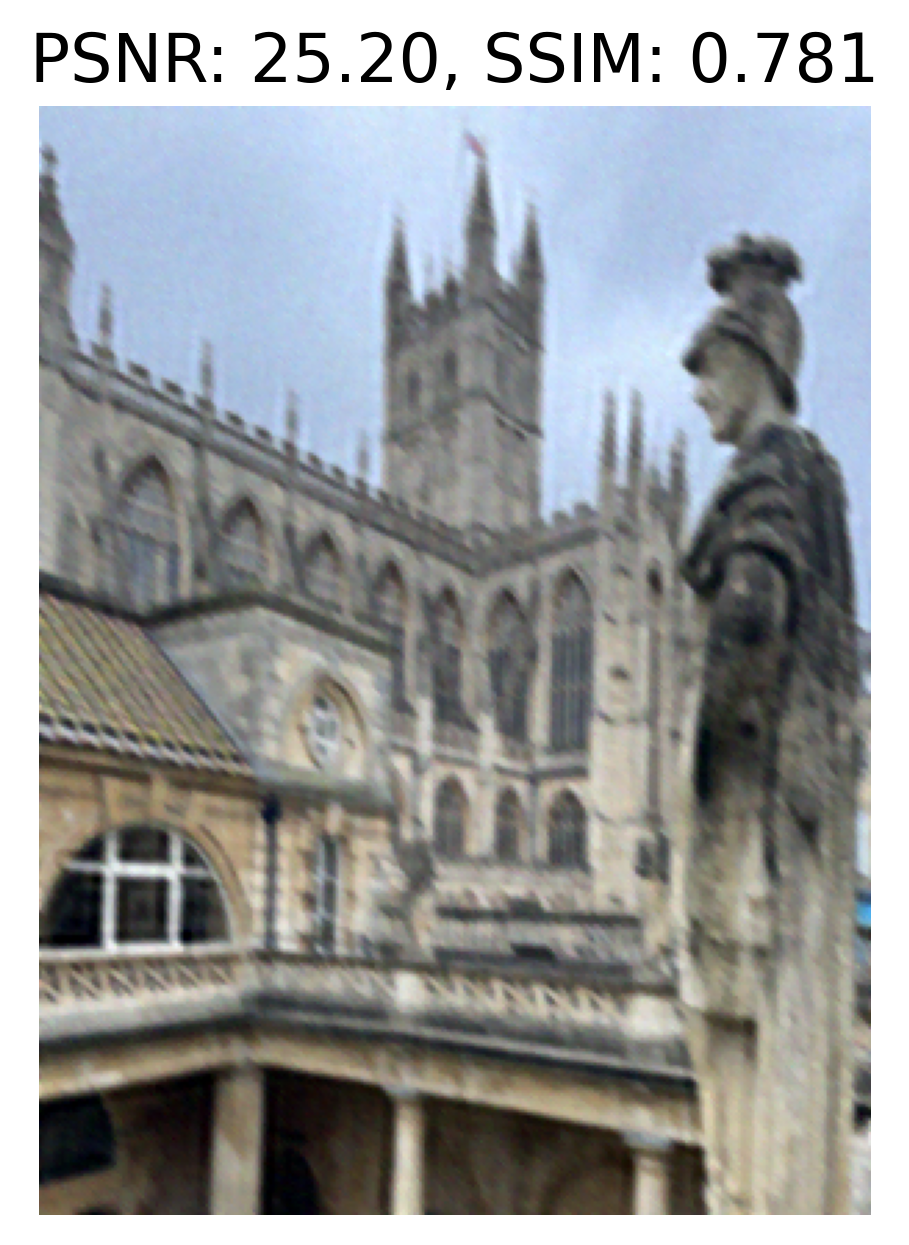}
  \end{subfigure}
  \hfill
  \begin{subfigure}[t]{0.24\textwidth}
    \includegraphics[width=\linewidth]{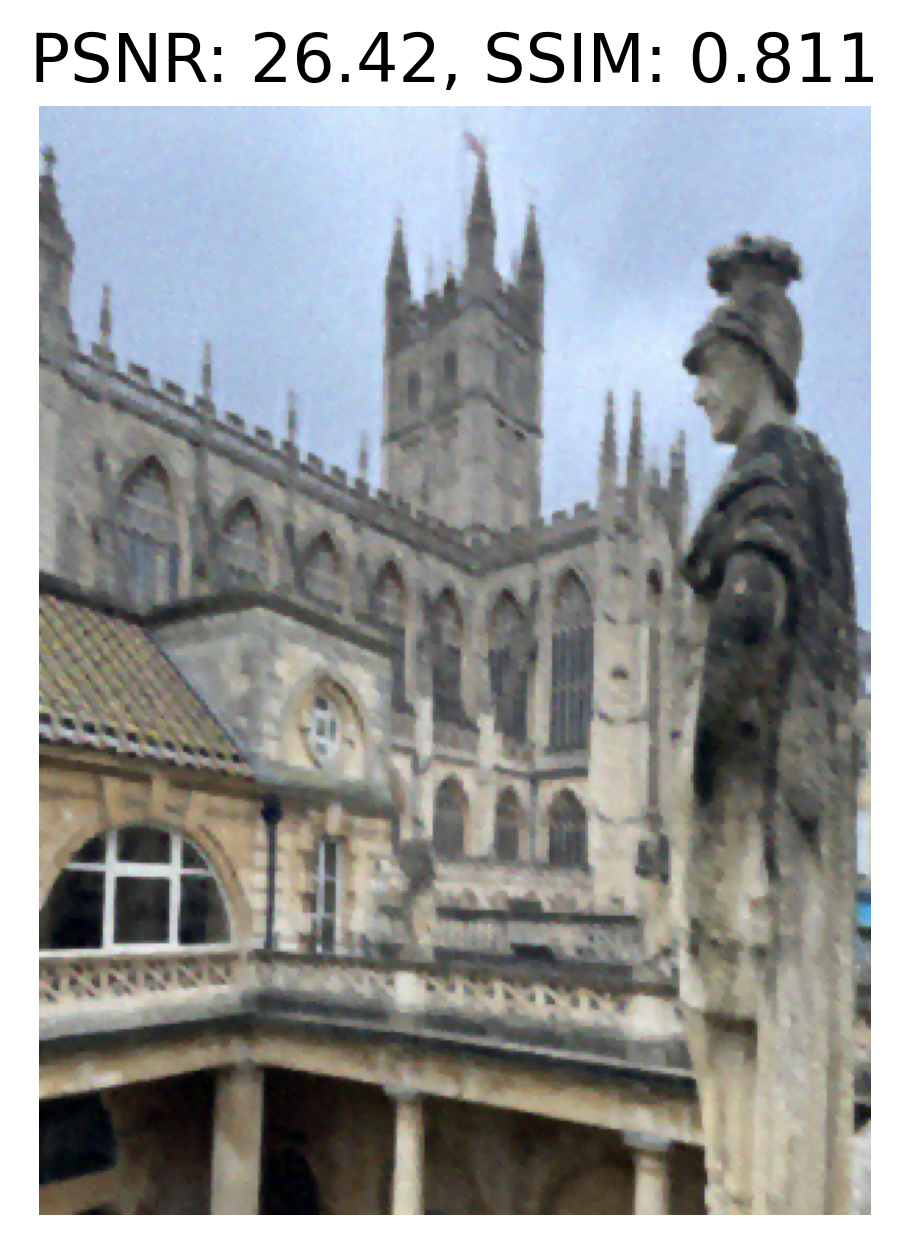}
  \end{subfigure}
  \vspace{1pt}  
  \begin{subfigure}[t]{0.24\textwidth}
    \includegraphics[width=\linewidth]{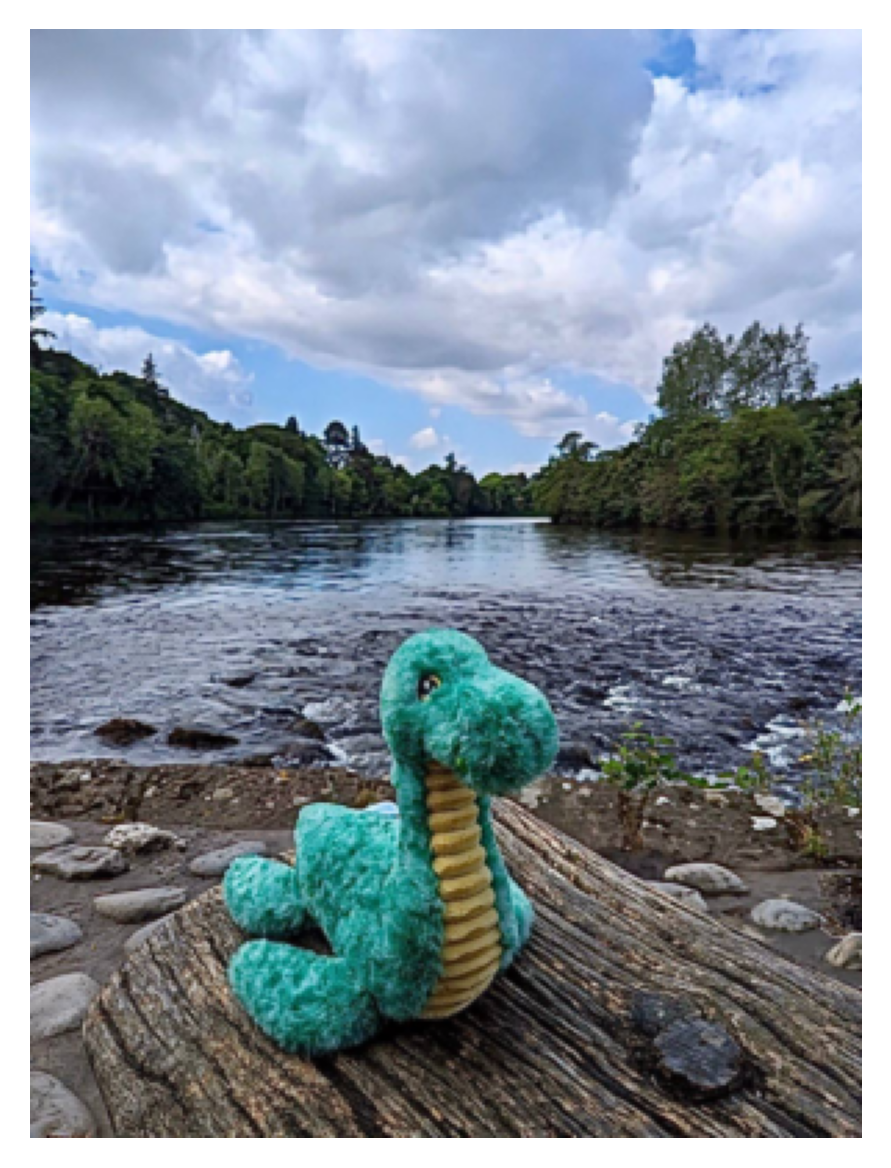}
    \caption{Ground truth}
  \end{subfigure}
  \hfill
  \begin{subfigure}[t]{0.24\textwidth}
    \includegraphics[width=\linewidth]{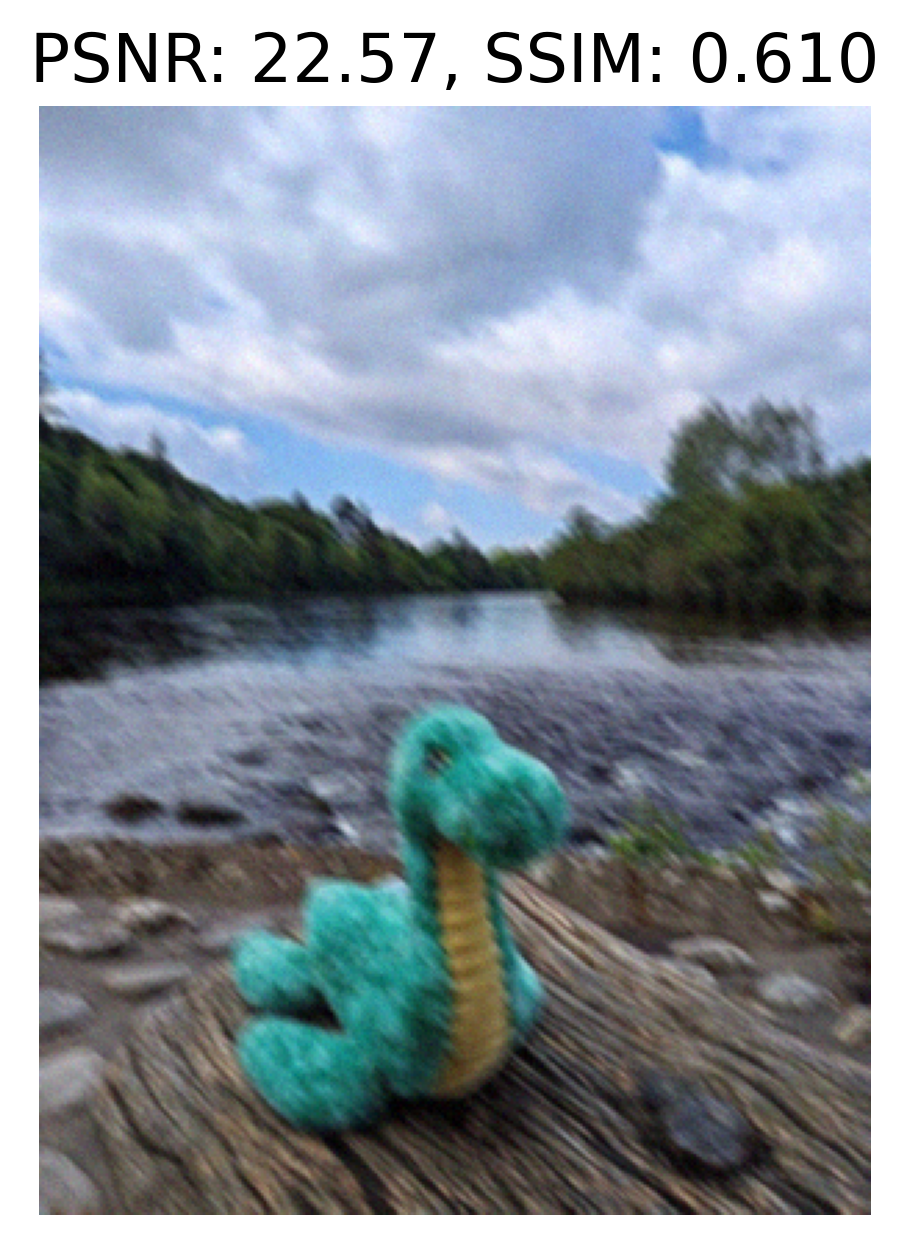}
    \caption{Blurry and noisy}
  \end{subfigure}
  \hfill
  \begin{subfigure}[t]{0.24\textwidth}
    \includegraphics[width=\linewidth]{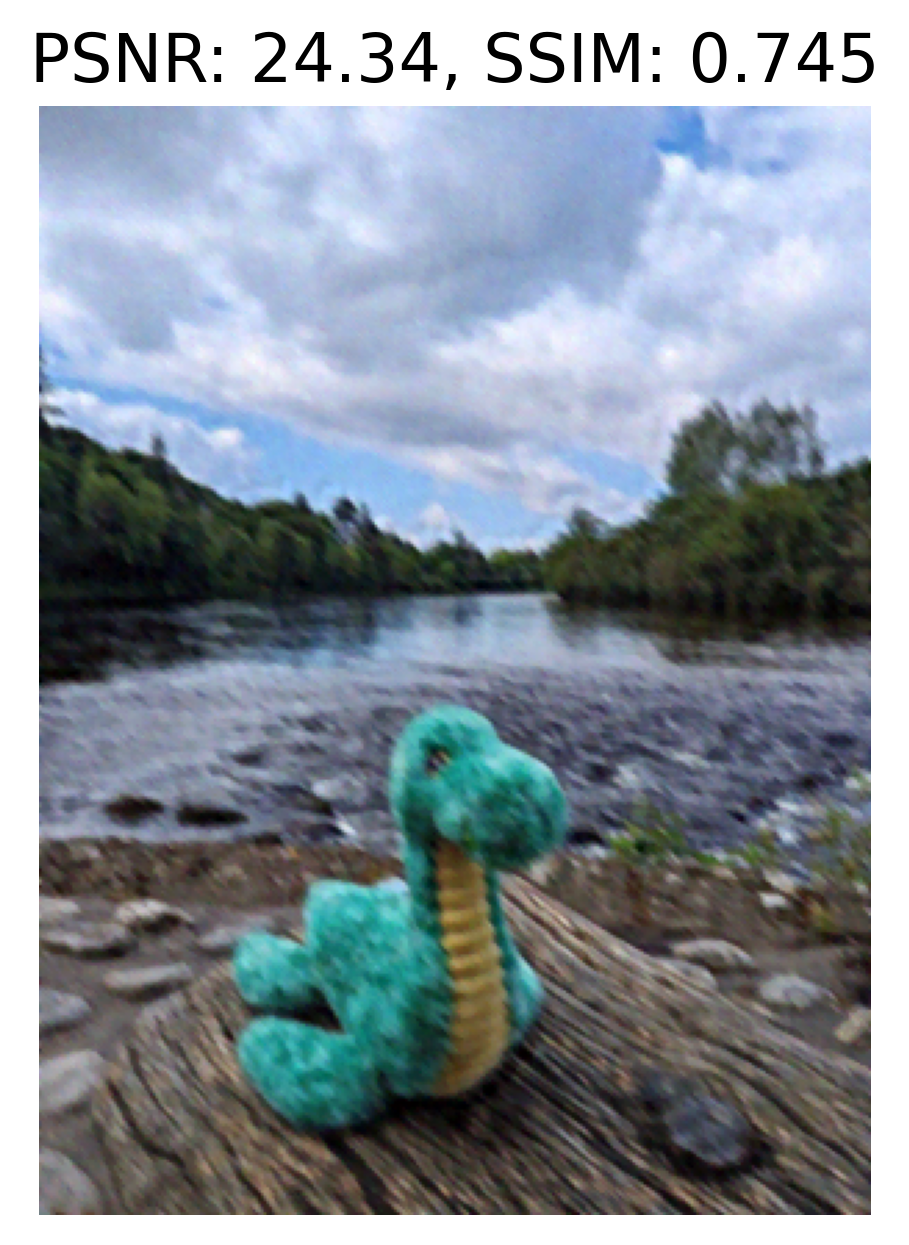}
    \caption{TV}
  \end{subfigure}
  \hfill
  \begin{subfigure}[t]{0.24\textwidth}
    \includegraphics[width=\linewidth]{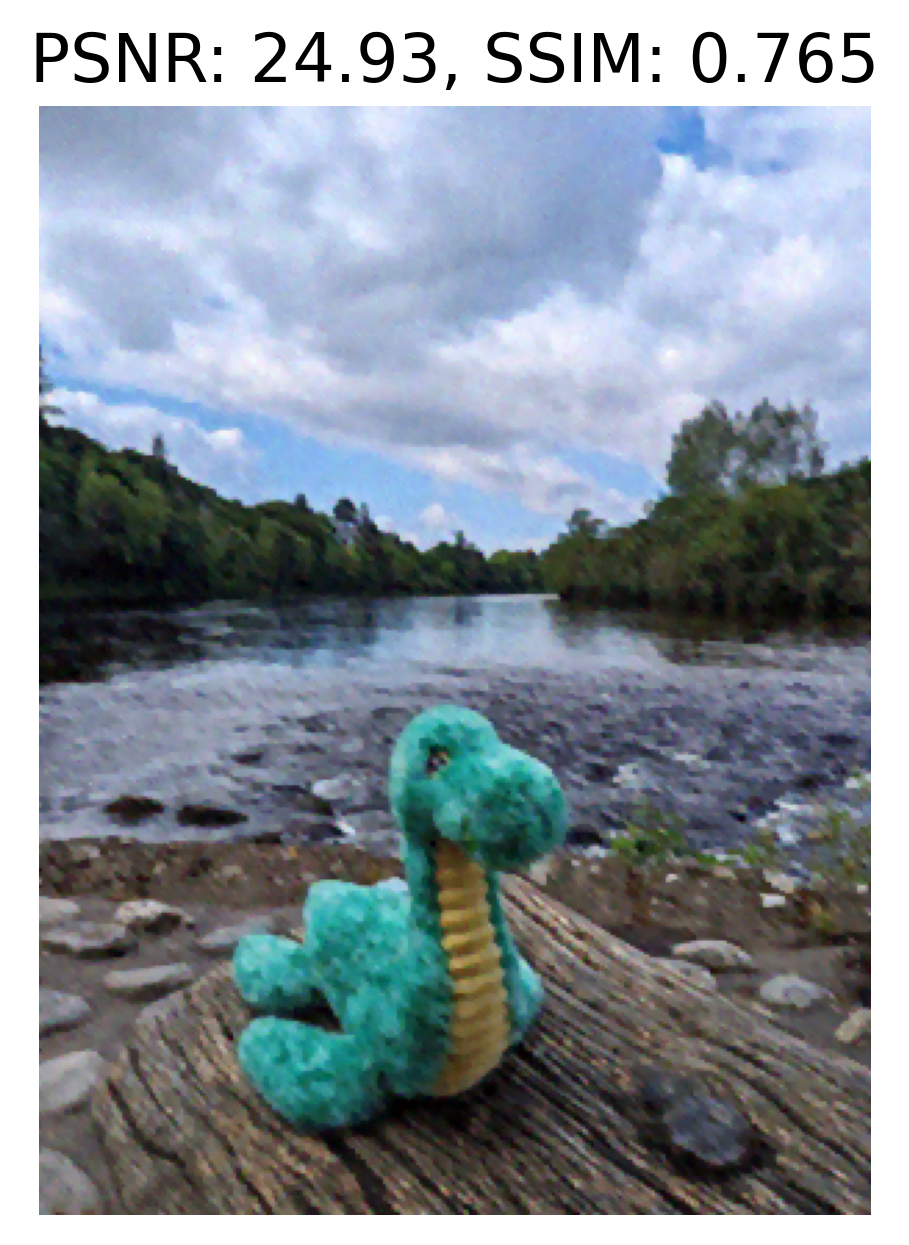}
    \caption{\Our{}}
  \end{subfigure}

  \caption{Reconstructions in semi-blind deblurring (initial kernel: motion blur).}
  \label{fig:all_mg}
\end{figure}

The deblurred image for this mixed blur is presented in \cref{fig:all_mg}, where the last column demonstrates that \Our{}  achieves higher reconstruction quality, both visually and in terms of
quality metrics, compared to  smoothed TV. Examining the changes in the (first channel of the) kernel (reported in {\cref{fig:kernel_semiblind}} and corresponding to the images in \cref{fig:all_mg}), reveals a Gaussian-like difference that \Our{} was able to recover successfully.

\begin{figure}[h!]
\centering
\begin{tabular}{@{}c@{\quad}c@{\quad}c@{}}
\includegraphics[width=0.3\textwidth]{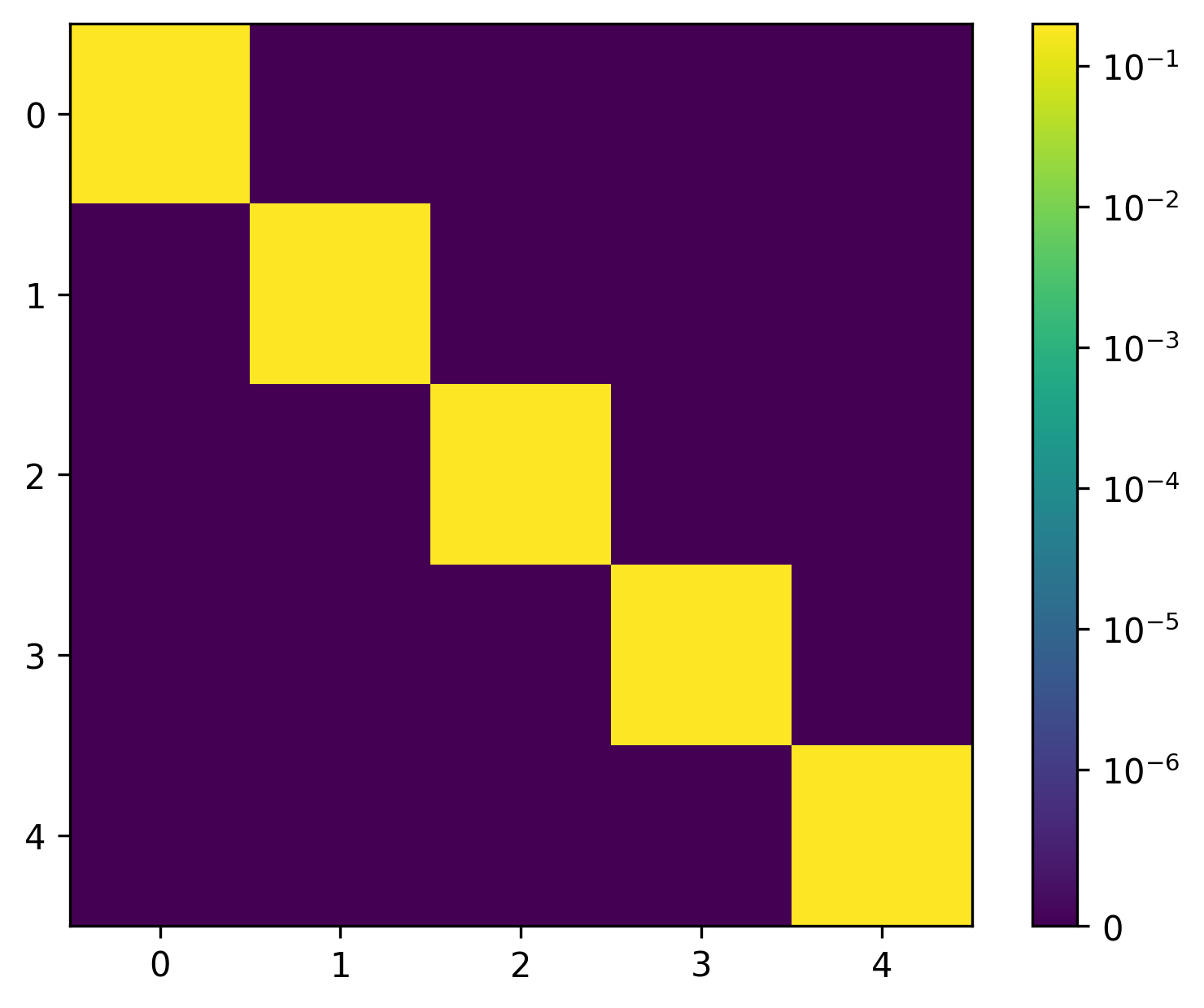}
& \includegraphics[width=0.3\textwidth]{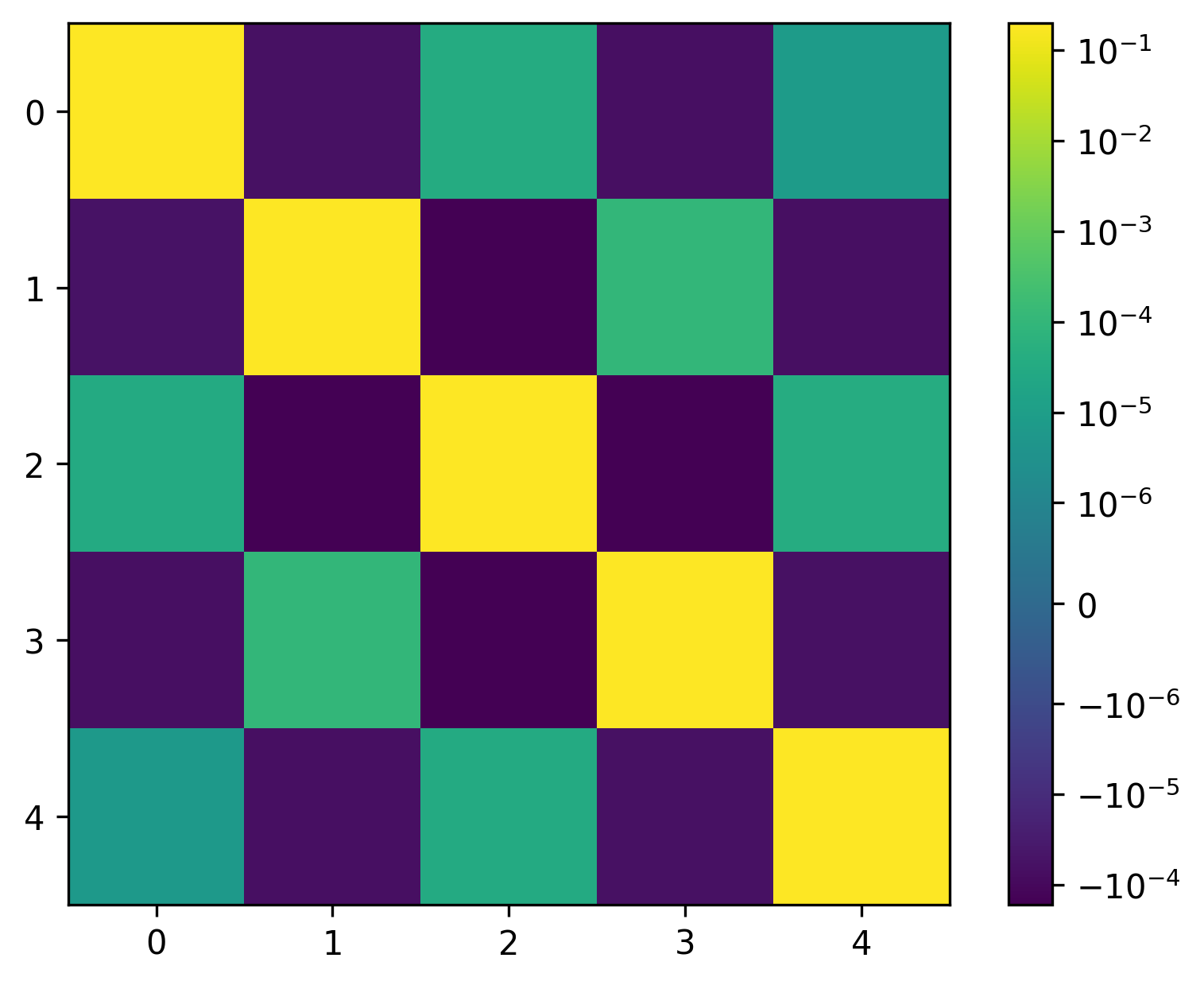} 
& \includegraphics[width=0.3\textwidth]{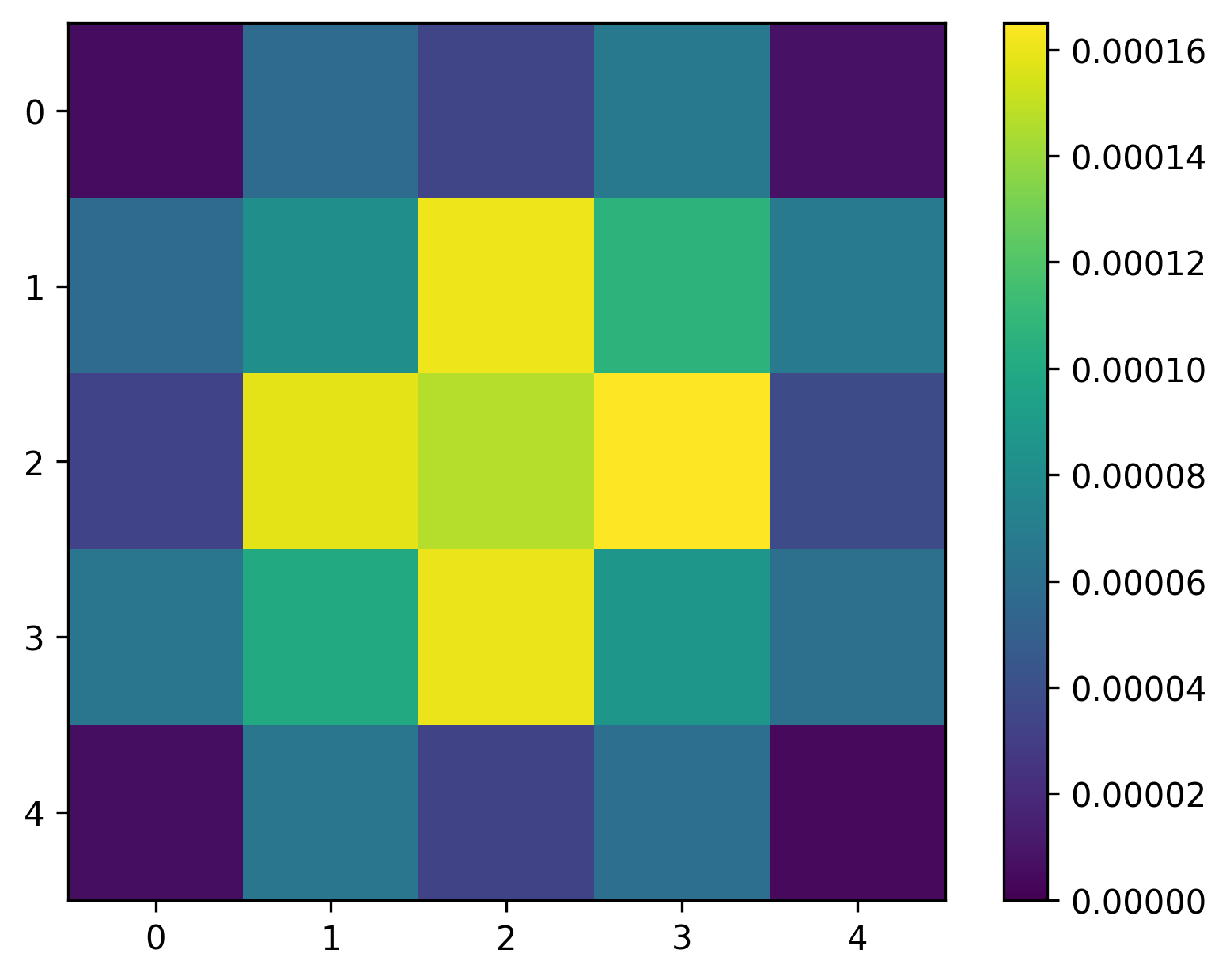} \\
\footnotesize{(a) Initial kernel $b_0$}  & 
\footnotesize{(b) Kernel $b^*$} &
\footnotesize{(c) $| b^* - b_0|$}
\end{tabular}
\caption{Kernels (first channel) of the 2D mixed blur forward operator. 
(a) Initial kernel $b_0$. (b) Optimal kernel $b^*$ recovered by \Our{}. (c) Difference $|b^*-b_0|$.}
    \label{fig:kernel_semiblind}
\end{figure}

\vspace{-5pt}

\section{Conclusions}
\label{sec:conclusions}
In this paper we discussed the applicability of an adaptive inexact bilevel optimization method (MAID) to the Sobolev-regularized analytical deep image prior problem. We have seen that  MAID  admits a straightforward generalization to an infinite-dimensional Hilbert-space setting, yet the non-differenti\-ability of the Sobolev regularization term in the $L^2$ setting requires further development of the MAID algorithm. In particular, an adaptive inexact version of the PALM algorithm seems desirable, which we plan to investigate in our future work. In the finite-dimensional setting, however, we demonstrate that  adaptive inexact bilevel optimization can achieve significant computational speed-ups and allows us to use ADP in larger-scale problems than in pervious literature.

\subsection*{Acknowledgements} 
This research made use of Hex, the GPU Cloud in the Department of Computer Science at the University of Bath.
MSS is supported by a scholarship from the EPSRC Centre for Doctoral Training in Statistical Applied Mathematics at Bath (SAMBa) under the project EP/S022945/1. 
TAB is a member of INdAM-GNCS and acknowledges support by the European Union - NextGeneration EU through the Italian Ministry of University and Research as part of the PNRR – M4C2, Investment 1.3 (MUR Directorial Decree no. 341 of 03/15/2022), FAIR ``Future Partnership Artificial Intelligence Research'', Proposal Code PE00000013 - CUP J33C22002830006).
YK acknowledges the support of the EPSRC (Fellowship EP/V003615/2 and Programme Grant EP/V026259/1).
%

\bibliographystyle{abbrv}
\bibliography{fullbib}

\end{document}